\documentclass[11pt]{article}

\usepackage{paralist}
\usepackage[margin=3.0cm]{geometry} 
\usepackage{graphicx} 

\usepackage{amssymb, amsfonts, amsmath, amscd, amsthm} 
\usepackage{mathtools} 
\usepackage{titlesec} 
\usepackage{csquotes} 
\usepackage[shortlabels]{enumitem} 
\usepackage{tikz-cd} 
\usepackage{wrapfig} 
\usepackage{epigraph} 
\usepackage{hyperref} 
\usepackage{todonotes} 
\usepackage{mathrsfs} 
\usepackage{mathtools}
\usepackage{todonotes}
\usepackage{soul}
\usepackage[normalem]{ulem}
\usepackage[misc]{ifsym}
\usepackage{multicol}

\hypersetup{
     colorlinks   = true,
     citecolor    = green
}

\numberwithin{equation}{section}

\titleformat*{\section}{\Large \scshape\center}
\titleformat*{\subsection}{\fontsize{14}{14} \sffamily}

\theoremstyle{plain}
\newtheorem{theorem}{Theorem}[section]
\newtheorem*{theorem*}{Theorem}

\newtheorem{lemma}[theorem]{Lemma}
\newtheorem{proposition}[theorem]{Proposition}
\newtheorem{corollary}[theorem]{Corollary}

\theoremstyle{definition}
\newtheorem{definition}[theorem]{Definition}
\newtheorem*{definition*}{Definition}
\newtheorem{example}[theorem]{Example}

\theoremstyle{remark}
\newtheorem{remark}{Remark}

\DeclareMathOperator{\tr}{tr}

\begin{document}
\pagenumbering{gobble}
\title{\huge{A Quantum Harmonic Analysis \\ Approach to Segal Algebras}}
\author{Eirik Berge, Stine Marie Berge, and Robert Fulsche}
\date{}
\maketitle
\pagenumbering{arabic}
\begin{abstract}
    In this article, we study a commutative Banach algebra structure on the space  $L^1(\mathbb{R}^{2n})\oplus \mathcal{T}^1$, where the $\mathcal{T}^1$ denotes the trace class operators on $L^2(\mathbb{R}^{n})$. The product of this space is given by the convolutions in quantum harmonic analysis. Towards this goal, we study the closed ideals of this space, and in particular its Gelfand theory. We additionally develop the concept of quantum Segal algebras as an analogue of Segal algebras. We prove that many of the properties of Segal algebras have transfers to quantum Segal algebras. However, it should be noted that in contrast to Segal algebras, quantum Segal algebras are not ideals of the ambient space. We also give examples of different constructions that yield quantum Segal algebras.
\end{abstract}
\section{Introduction}
Classical mechanics of a particle on $\mathbb{R}^n$ is often described by endowing a symplectic structure on its phase space $\mathbb{R}^{2n}$. Given a point $(p,q)\in \mathbb{R}^{2n}$, the point $p$ describes the position and $q$ describes the momentum of a particle.
In quantum physics, functions on $\mathbb{R}^{2n}$ are typically referred to as \emph{classical observables}, while (not necessarily bounded) operators on $\mathcal H=L^2(\mathbb{R}^n)$ are referred to as \emph{quantum observables}.
The Weyl's canonical commutation relations on $\mathcal H$ plays the same role as the symplectic phase space structure in classical mechanics. Usually, the theory for the classical observables and the quantum observables are considered on different spaces, with quantization acting as a way to go between classical and quantum observables. In his 1984 paper \cite{werner84}, R.~Werner introduced a framework for the simultaneous treatment of both the classical and the quantum observables, which he called \emph{quantum harmonic analysis}. Here, he extends notions in harmonic analysis for the classical observables, such as convolutions and Fourier transforms, to the quantum observables. In doing so, one achieves interesting interactions between the classical and the quantum side of the theory. Besides being an object of interest for mathematical physics and also of intrinsic interest, this framework also allows for important applications in several fields, e.g., localization operators, Wigner distributions and Toeplitz operators. Reformulating problems in terms of quantum harmonic analysis offers a new framework for tackling problems, which often proves fruitful. For more information, see e.g., \cite{Berge_Berge_Luef_Skrettingland2022, Fulsche2020, Fulsche2022, Halvdansson2022, luef_eirik2018, Luef_Skrettingland2021, Sk20, Werner1983}.

Being a bit more precise, let $W_z$ denote the \emph{Weyl operators} on the Hilbert space $\mathcal H$ defined by
\begin{align*}
W_{(x,\,\xi)}\phi(y)=e^{i\xi y-ix\xi/2}\phi(y-x), \quad z=(x,\,\xi)\in \mathbb{R}^{2n}.
\end{align*}
As is common, we have set Planck's constant to $\hbar =1$. 
The Weyl operators are unitary operators on $\mathcal{H}$, and satisfies the \emph{Weyl commutation relations}
\begin{align*}
    W_z W_{z'}=e^{-i\sigma(z,\,z')/2}W_{z+z'} = e^{-i\sigma(z,\,z')} W_{z'} W_z, \quad z, z' \in \mathbb R^{2n},
\end{align*}
where $\sigma$ denotes the standard symplectic form on $\mathbb{R}^{2n}$.
Given two trace class operators $S$ and $T$, their convolution is defined by
\begin{align*}
    S \ast T(z) \coloneqq \tr(S W_z P T PW_{-z}), \quad z \in \mathbb R^{2n},
\end{align*}
where $P$ is the \emph{parity operator} acting on $\mathcal{H}$ by $P\phi(x) = \phi(-x)$.
The expression $\alpha_z(A) \coloneqq W_z A W_{-z}$ plays the role of \emph{shifting} the operator $A$ by $z$, analogous to the shift $\alpha_z(f) = f(\cdot -z )$ used in the definition of function convolution.
Defining the \emph{Fourier-Weyl transform} (also called \emph{Fourier-Wigner transform}) by 
\begin{align*}
    \mathcal{F}_W(S)(z) \coloneqq \tr(S W_z), \quad z \in \mathbb R^{2n},
\end{align*}
we have that 
\begin{align*}
    \mathcal{F}_\sigma(S\ast T) = \mathcal{F}_W(S)\mathcal{F}_W(T),
\end{align*}
where $\mathcal{F}_\sigma$ is the symplectic Fourier transform. Additionally, given a function $f\in L^1(\mathbb{R}^{2n})$ and a trace class operator $A\in \mathcal{T}^1(\mathcal{H})$, one defines the operator $f \ast A$ by
\begin{equation*}
    (f\ast A)\phi = \frac{1}{(2\pi)^n}\int_{\mathbb{R}^{2n}}f(z)W_z AW_{-z}\phi\,\mathrm{d}z, \quad z \in \mathbb R^{2n}.
\end{equation*}
Again, we have a relation with the Fourier transforms given by 
\[\mathcal{F}_W(f\ast A)=\mathcal{F}_\sigma(f)\mathcal{F}_W(A).\]
 Using the convolutions, one can define a product on $L^1(\mathbb{R}^{2n})\oplus\mathcal{T}^1(\mathcal H)$ by \[(f,S)\ast (g,T)=(f\ast g+S\ast T, f\ast T+g\ast S),\] making the space into a commutative Banach algebra. Hence, this new space contains both the classical and the quantum observables.

It is not hard to find reasons why $L^2(\mathbb R^{2n})$ and the Hilbert-Schmidt operators $\mathcal T^2(\mathcal H)$ play similar roles for classical and quantum observables. For example, the Fourier-Weyl transform can be extended to an isomorphism between these two spaces by the Plancherel theorem, see e.g., \cite[Sec.~7.5]{folland2016course}. For other function spaces on $\mathbb{R}^{2n}$, the paper \cite{werner84} gives a definition of \emph{corresponding spaces}. We say that a space of classical observables $\mathcal{D}_0$ and a space of quantum observables $\mathcal{D}_1$ are corresponding spaces if 
\begin{equation}\label{eq:corresponding spaces}
 (L^1(\mathbb{R}^{2n})\oplus\mathcal{T}^1(\mathcal H))\ast (\mathcal{D}_0\oplus\mathcal{D}_1)\subset \mathcal{D}_0\oplus\mathcal{D}_1.
\end{equation}
Using this definition $L^2(\mathbb{R}^{2n})$ and $\mathcal{T}^2(\mathcal{H})$ are corresponding spaces, after extending the definition of convolution appropriately, the same is true for $L^\infty(\mathbb R^{2n})$ and the space of bounded operators $\mathcal L(\mathcal H)$. It should be noted that this correspondence only becomes unique, in the sense that given a space $\mathcal{D}_0$ there is only one corresponding space $\mathcal{D}_1$, upon assuming additional topological properties on $\mathcal D_0$ and $\mathcal D_1$. Notice that when $\mathcal{D}_0\oplus\mathcal{D}_1\subset L^1(\mathbb{R}^{2n})\oplus\mathcal{T}^1(\mathcal H)$, equation \eqref{eq:corresponding spaces} implies that $\mathcal{D}_0\oplus \mathcal{D}_1$ is an ideal.

This paper has the dual purpose of both furthering our understanding of the space $L^1(\mathbb R^{2n}) \oplus \mathcal T^1(\mathcal H)$ as a commutative Banach algebra and introducing a quantum version of Segal algebras.
For the first goal, we completely classify the closed ideals of $L^1(\mathbb R^{2n}) \oplus \mathcal T^1(\mathcal H)$. As is well-known, closed ideals of $L^1(\mathbb{R}^n)$ are identical to closed shift-invariant subspaces. This analogy breaks down for closed ideals of $L^1(\mathbb R^{2n}) \oplus \mathcal T^1(\mathcal H)$, where 
shift-invariance is merely a necessary condition for being a closed ideal, but not sufficient. Nevertheless, there is still a rich connection between shift-invariant subspaces and closed ideals, which we investigate. 
As a special case, we describe the Gelfand theory of $L^1(\mathbb R^{2n}) \oplus \mathcal T^1(\mathcal H)$. From the Gelfand theory of $L^1(\mathbb{R}^{2n})$ it should come as no surprise that this is related to the Fourier transforms.  

The second purpose of the paper is to introduce a notion of \emph{quantum Segal algebras}. The quantum Segal algebras play an analogous role in quantum harmonic analysis to the Segal algebras in classical harmonic analysis. The definition of a quantum Segal algebra is in spirit the same as Segal algebras: A quantum Segal algebra is a dense subalgebra of $L^1(\mathbb R^{2n}) \oplus \mathcal T^1(\mathcal H)$ that comes with its own Banach algebra norm such that shifts act isometrically and continuously on it. We show that quantum Segal algebras are, in analogy with Segal algebras, the same as essential $L^1$ modules. Further, we prove that under the technical assumption of the space being \emph{graded} (i.e., the algebra has the structure of a direct sum of its classical and quantum part), they always have the same Gelfand theory as $L^1( \mathbb R^{2n}) \oplus \mathcal T^1(\mathcal H)$. As a last general result on quantum Segal algebras, we give a description of the intersection of all graded quantum Segal algebras.

After having discussed the general theory of quantum Segal algebras, we turn towards examples of such algebras. We give two different ways of constructing such algebras: By convolving a Segal algebra with a regular trace class operator or by using the Weyl quantization. As a special case of a quantum Segal algebra, we discuss the \emph{quantum Feichtinger algebra}. It is obtained as the direct sum of the Feichtinger algebra $\mathcal S_0(\mathbb R^{2n})$ and the subspace of $\mathcal T^1(\mathcal H)$ obtained by the Weyl quantization of symbols in $\mathcal S_0(\mathbb R^{2n})$. Due to the importance of the classical Feichtinger algebra, we find it appropriate to look into this example in some more detail. In particular, we give some conditions on when an operator belongs to the quantum Feichtinger algebra in terms of the operations of quantum harmonic analysis. Not by chance, this turns out to be entirely analogous to the characterization of the membership of functions in the classical Feichtinger algebra. 

Our paper is structured as follows: In Section \ref{sec:preliminaries}, we give a more detailed account of the foundations for our work. We start by recalling the basics of Segal algebras to thereafter review some results in quantum harmonic analysis. In particular, we introduce the notion of \emph{modulation} of an operator, which when applying the Fourier Weyl transform is turned into function shifts. This notion of modulation seems not to be present in the literature thus far. In Section \ref{sec: structure theory}, we conduct the investigation of the closed ideals of $L^1(\mathbb R^{2n}) \oplus \mathcal T^1(\mathcal H)$ as well as the Gelfand theory. Section \ref{sec: Quantum Segal} is then devoted to the introduction and the basic properties of quantum Segal algebras. Finally, Section \ref{sec: examples of QSA} will discuss the abovementioned constructions of quantum Segal algebras, in particular the quantum Feichtinger algebra.

\section{Preliminaries}\label{sec:preliminaries}

\subsection{Classical Segal Algebras}
Before introducing Segal algebras, let us quickly recall some basic facts about the space of integrable functions $L^1(\mathbb{R}^n)$. We will let $\alpha_x\colon L^1(\mathbb{R}^n)\to L^1(\mathbb{R}^n)$ denote the \emph{shift operator} 
\[\alpha_xf(y)\coloneqq f(y-x), \qquad f \in L^{1}(\mathbb{R}^{n}), \  x\in \mathbb{R}^n.\] 
For elements $f, g \in L^{1}(\mathbb{R}^{n})$ their \emph{convolution product} is
\[(f \ast  g)(x) \coloneqq \frac{1}{(2\pi)^{n/2}}\int_{\mathbb{R}^{n}}f(y)\alpha_{y}g(x) \, \mathrm{d}y = \frac{1}{(2\pi)^{n/2}}\int_{\mathbb{R}^{n}}f(y)g(x - y) \, \mathrm{d}y,\]
making $L^1(\mathbb{R}^n)$ into a Banach algebra. We want to note that the factors $\frac{1}{(2\pi)^{n/2}}$ serve as appropriate normalizations for the Haar measure, which makes the connection between classical and quantum harmonic analysis most natural.

The algebra $L^{1}(\mathbb{R}^{n})$ is equipped with the \emph{involution} \[f^{\ast }(x) \coloneqq \overline{f(-x)}.\] 
We use the notation $\mathcal{F}f$ for the \emph{standard Fourier transform} of $f \in L^{1}(\mathbb{R}^{n})$ given by 
\[\mathcal{F}f(\xi) \coloneqq \frac{1}{(2\pi)^{n/2}}\int_{\mathbb{R}^{n}}f(t)e^{- i t \xi} \, \mathrm{d}t.\]
We emphasize at this point that the above standard Fourier transform will not play a role later in this paper. It is usually replaced by the \emph{symplectic Fourier transform}, and we stress already at this point that $\widehat{f}$ will not denote the standard Fourier transform of $f$ as above, but its symplectic Fourier transform $\mathcal F_\sigma(f)$, see \eqref{eq:symplectic}.

Let us recall the definition of a Segal algebra, as e.g., given in \cite{reiter71}.
\begin{definition}[Segal Algebra]
A \emph{Segal algebra} $S$ is a linear subspace of $L^{1}(\mathbb{R}^{n})$ satisfying the following four properties:
\begin{enumerate}[(S1)]
\item\label{def:s1} The space $S$ is dense in $L^1(\mathbb R^n)$.
\item\label{def:s2} The space $S$ is a Banach algebra with respect to a norm $\| \cdot \|_S$ and convolution as the product.
\item\label{def:s3} The space $S$ is invariant under shifts, i.e., for $f\in S$ we have that $\alpha_x(f)\in S$ for all $x \in \mathbb{R}^{n}$.
\item\label{def:s4} The shifts $\alpha_x$ are continuous and satisfy $\|\alpha_{x}(f)\|_{S} = \|f\|_{S}$ for all $x \in \mathbb{R}^{n}$ and $f \in S$.
\end{enumerate}
\end{definition}
To check the continuity condition in property \ref{def:s4} one only need to check that given an arbitrary $f \in S$ and $\epsilon > 0$ there exists a neighborhood $U$ of the origin in $\mathbb{R}^{n}$ such that 
\[\|\alpha_{x}f - f\|_{S} < \epsilon \quad \text{ for } \quad  x \in U.\]
\begin{remark}
Segal algebras are often defined on a locally compact group $G$ as a dense subspace of $L^{1}(G)$ equipped with the left Haar measure and left shifts. In this more general case, the existence and variety of Segal algebras depend heavily on the group $G$ in question. For discrete groups the only Segal algebra is the whole space $l^{1}(G) = L^{1}(G)$.
On the other hand, for $G$ compact the space $L^{p}(G)$ is a Segal algebra for every $1\le p\le \infty$. In particular, the Segal algebra $L^2(G)$ is heavily used in Peter-Weyl theory.
\end{remark}

A well known result is that Segal algebras are (two-sided) ideals of $L^1(\mathbb{R}^n)$, i.e., for a Segal algebra $S$, we have that $f \ast g \in S$ whenever $f \in L^1(\mathbb R^n)$ and $g \in S$. Moreover, Segal algebras always satisfy by \cite[Prop.\ 4.1]{reiter71} the estimate
\begin{equation}\label{eq:banach_ideal_bound}
    \|f \ast  g\|_{S} \leq C \|f\|_{L^{1}(\mathbb{R}^{n})} \|g\|_{S},
\end{equation}
for $f \in L^{1}(\mathbb{R}^{n})$ and $g \in S$ for a $C$ only dependent on $\|\cdot\|_{S}$. By rescaling the norm on $S$ to an equivalent norm, one can take the constant $C$ to be $1$. More generally, one has that if $T\subset S$ are both Segal algebras then
\begin{equation}\label{eq:banach_ideal_bound_general}
    \|f \ast  g\|_{T} \leq C \|f\|_{S} \|g\|_{T}\qquad \ f\in S\,,  g\in T.
\end{equation}
We refer the reader to \cite[Prop.\ 4.2]{reiter71} for a generalization of \eqref{eq:banach_ideal_bound} involving bounded measures.

\begin{definition}
We say that a Segal algebra $S$ is \emph{star-symmetric} if the involution is an isometry in the $S$-norm, meaning that for $f\in S$ we have that $f^\ast \in S$ and $\| f^\ast\|_S = \| f\|_S.$
\end{definition}

Clearly $L^{1}(\mathbb{R}^{n})$ is a star-symmetric Segal algebra. In the following example, we will review a few classical examples of Segal algebras of central importance.

\begin{example}\label{ex:segal}\hfill
  \begin{enumerate}[(E1)]
      \item Let $C_{0}(\mathbb{R}^{n})$ denote the continuous function on $\mathbb{R}^{n}$ that vanish towards infinity. Then the space $S_{\infty} \coloneqq L^{1}(\mathbb{R}^{n}) \cap C_{0}(\mathbb{R}^{n})$ is a star-symmetric Segal algebra with the norm 
      \[\|f\|_{S_{\infty}} \coloneqq \|f\|_{L^{1}(\mathbb{R}^{n})} + \|f\|_{L^\infty(\mathbb{R}^{n})}.\]
      \item For $1 < p < \infty$ the space $S_{p} \coloneqq L^{1}(\mathbb{R}^{n}) \cap L^{p}(\mathbb{R}^{n})$ is a star-symmetric Segal algebra with the norm 
      \[\|f\|_{S_{p}} \coloneqq \|f\|_{L^{1}(\mathbb{R}^{n})} + \|f\|_{L^{p}(\mathbb{R}^{n})}.\]
      \item\label{ex:segal E3} Fix a positive and unbounded measure $\mu$ on $\mathbb{R}^{n}$, for example the Lebesgue measure $\mathrm{d}x$. For $1 \leq p < \infty$ we denote by \[S_{p}^{\mu} \coloneqq L^{1}(\mathbb{R}^{n}) \cap \mathcal{F}^{-1}\left(L^{p}(\mathbb{R}^{n}, \mu)\right).\] Then $S_{p}^{\mu}$ is a star-symmetric Segal algebra with the norm 
      \[\|f\|_{S_{p}^{\mu}} \coloneqq \|f\|_{L^{1}(\mathbb{R}^{n})} + \|\mathcal{F}(f)\|_{L^{p}(\mathbb{R}^{n}, \mu)}.\]
      \item Recall that \emph{Wiener's algebra} $W$ on $\mathbb R^{n}$ is defined as the space of those continuous functions on $\mathbb R^{n}$ such that
\begin{align*}
    \| f\|_W' \coloneqq \sum_{m \in \mathbb Z^{n}} a_m(f) < \infty, \quad \text{where } a_m(f) = \sup_{x \in [0,1]^{n}} |f(m + x)|.
\end{align*}
Clearly, $W \subset L^1(\mathbb R^{n}) \cap C_0(\mathbb R^{n})$. Letting the norm on $W$ be defined by
\begin{align*}
    \| f\|_W \coloneqq \sup_{y \in \mathbb R^{n}} \| \alpha_y(f)\|_W',
\end{align*}
turns $W$ into a Segal algebra. It is immediately clear that $W$ is also a module under pointwise multiplication by functions from $C_0(\mathbb R^{n})$. Indeed, $W$ is the smallest Segal with this property:
\begin{theorem}[\cite{feichtinger1977}]
    Let $ S \subset L^1(\mathbb R^n)$ be a Segal algebra which is a $C_0(\mathbb R^n)$ module under pointwise multiplication. Then, $S$ contains $W$ and $\| f\|_W \leq C\| f\|_S$ for all $f\in W$.
\end{theorem}
      \item\label{ex:Feichtinger} The \emph{Feichtinger algebra} $\mathcal{S}_{0}(\mathbb R^n)$ in time-frequency analysis is typically defined through the short-time Fourier transform: One can define the Feichtinger algebra as the set of functions $f \in L^{1}(\mathbb{R}^{n}) \cap L^{2}(\mathbb{R}^{n})$ such that $V_{f}f \in L^{1}(\mathbb{R}^{2n})$, where
      \begin{equation}
      \label{eq:STFT}
          V_{g}f(x, \xi) \coloneqq \int_{\mathbb{R}^{n}}f(t)\overline{g(t - x)}e^{- i \xi t} \, \mathrm{d}t.
      \end{equation}
      Through this description, it is somewhat nontrivial to see that the Feichtinger algebra is a star-symmetric Segal algebra. However, one can also describe the Feichtinger algebra as the minimal Segal algebra that is (in a strong sense) closed under \emph{modulations}
      \begin{equation}\label{eq:modulation}
      M_{\xi}f(t) \coloneqq e^{i t \xi}f(t).
      \end{equation}
      We refer the reader to \cite{jakobsen18} for a careful examination of the Feichtinger algebra on locally compact groups.
  \end{enumerate}
\end{example}

\begin{remark}\label{re:intersection_Segal}
The intersection of all Segal algebras is precisely the set of continuous functions in $L^{1}(\mathbb{R}^{n})$ whose Fourier transform is compactly supported, we refer to \cite[(vii) p.~26]{reiter71} for a proof. As such, not every Segal algebra contains the Schwartz functions $\mathcal{S}(\mathbb{R}^{n})$ as a subspace.
\end{remark}

Segal algebras on $\mathbb{R}^{n}$ are not unital algebras, since $(2\pi)^{n/2}\delta_0\not\in L^1(\mathbb{R}^n)$. However, every Segal algebra  contains an approximate unit that is normalized in $L^{1}(\mathbb{R}^{n})$, see \cite[Prop.~8.1]{reiter71}. As such, the Cohen-Hewitt factorization theorem implies that every element in a Segal algebra $g \in S$ can be factorized as $g = f \ast  h$ for $f \in L^{1}(\mathbb{R}^{n})$ and $h \in S$. We succinctly write 
\begin{equation}\label{eq:Cohen_Hewitt_factorization}
    L^{1}(\mathbb{R}^{n}) \ast  S = S.
\end{equation} 
\par 
It follows from \cite[Prop.~3.1]{dales2014} and \eqref{eq:Cohen_Hewitt_factorization} applied to $S = L^{1}(\mathbb{R}^{n})$ that no maximal ideal of $L^{1}(\mathbb{R}^{n})$ can be dense. Hence Segal algebras are never maximal ideals. We note that $L^{1}(\mathbb{R}^{n})$ has plenty of closed maximal ideals, e.g., for every $\xi \in \mathbb{R}^{n}$ the space
\[I_{\xi} \coloneqq \left\{f \in L^{1}(\mathbb{R}^{n}) : \mathcal{F}f(\xi) = 0\right\},\] 
is a maximal ideal.
When it comes to ideals inside a Segal algebra $S \subset L^{1}(\mathbb{R}^{n})$, then every closed ideal $I_{S}$ in $S$ is on the form $I_{S} = I \cap S$, where $I$ is a unique closed ideal of $L^{1}(\mathbb{R}^{n})$ by \cite[Thm.~9.1]{reiter71}.

\subsection{Quantum Harmonic Analysis}
The operators in this article will always be on the Hilbert space $\mathcal H=L^2(\mathbb R^n)$. Recall that the \emph{Weyl operators} acting on $\mathcal H$ are given by
\begin{align*}
W_{(x,\,\xi)}\phi(y)\coloneqq e^{i\xi y-ix\xi/2}\phi(y-x),
\end{align*}
where $x, \, \xi \in \mathbb R^n$.
It is common to use the shorthand notation $W_z\coloneqq W_{(x,\,\xi)}$, where $z = (x,\, \xi)\in \mathbb{R}^{2n}$. The Weyl operators are unitary operators and satisfy $W_{z}^\ast = W_{-z}$. The name Weyl operators comes from the \emph{Weyl commutation relation}, often called the CCR relation, given by
\begin{align*}
    W_zW_{z'}=e^{-i\sigma(z,\,z')/2}W_{z+z'}=e^{-i\sigma(z,\,z')}W_{z'}W_{z}.
\end{align*}
Here, $\sigma$ denotes the symplectic form 
\begin{align*}
   \sigma(z, z')= \sigma((x, \xi), (x', \xi')) \coloneqq \xi x' - x \xi'.
\end{align*}
We will now define the basic operations of quantum harmonic analysis, namely convolutions and Fourier transforms. Unless otherwise stated, the proof of the statements in Section \ref{sec: QHA convolutions} and Section \ref{sec:qha fourier} can be found in \cite{werner84}.

\subsubsection{Convolutions}\label{sec: QHA convolutions}
Given a bounded linear operator $A \in \mathcal L(\mathcal H)$, we define \emph{shifting the operator} $A$ by $ z  \in \mathbb R^{2n}$ as
\[\alpha_z(A) \coloneqq W_z A W_{-z}.\]
Using shifts, one can define the \emph{function-operator convolution} for $f \in L^1(\mathbb R^{2n})$ and $A \in \mathcal L(\mathcal H)$ as the operator
\begin{equation}\label{eq:function_operator_convolution}
    f \ast A \coloneqq 
    \frac{1}{(2\pi)^n}\int_{\mathbb R^{2n}} f(z) \alpha_z(A)\,\mathrm{d}z\eqqcolon A\ast f.
\end{equation}
The convolution \eqref{eq:function_operator_convolution} is interpreted as a weak integral and the resulting operator is bounded. We say that a linear space $X\subset \mathcal{L} (\mathcal H)$ is \emph{shift-invariant} if $\alpha_z (X)\subset X$ for all $z\in \mathbb{R}^{2n}$. A Banach space $(X,\|\cdot \|_X)$ of operators $X\subset \mathcal L(\mathcal H)$ is said to have a \emph{strongly shift-invariant norm} if $\| \alpha_z(A)\|_X = \| A\|_X$ and $z \mapsto \alpha_z(A)$ is continuous in $\|\cdot\|_X$ for all $z\in \mathbb{R}^{2n}$. Given a shift-invariant Banach space with a strongly shift-invariant norm $(X,\|\cdot \|_X)$ continuously embedded into $\mathcal L(\mathcal H)$ then $f \ast A \in X$ for all $A\in X$ and $f \in L^1(\mathbb{R}^{2n})$.
Examples of such spaces are the \emph{$p$-Schatten class operators} $\mathcal T^p = \mathcal T^p(\mathcal H)$ for $1 \leq p \le \infty$. The special case $p=1$ is the \emph{trace class operators}, while $p=2$ is the \emph{Hilbert-Schmidt operators}. We use the standard convention that $\mathcal{T}^\infty$ is the compact operators with the operator norm.

We will use $P\colon \mathcal H\to \mathcal H$ to denote the \emph{parity operator} defined by $P\phi(x) = \phi(-x)$. Notice that the parity operator satisfies 
\[PW_z=W_{-z}P. \] For $A, B \in \mathcal T^1$ we define their \emph{operator-operator convolution} as
\begin{align*}
    A \ast B(z) \coloneqq \tr(A \alpha_z( P B P))=B \ast A(z), \quad z \in \mathbb R^{2n}.
\end{align*}
 It is of central importance in quantum harmonic analysis that $A \ast B \in L^1(\mathbb R^{2n})$. Furthermore, one has
\begin{align*}
   \frac{1}{(2\pi)^n} \int_{\mathbb R^{2n}} A \ast B(z) \,\mathrm{d}z = \tr(A) \tr(B).
\end{align*}
The convolutions satisfy the following associativity conditions
\[f\ast (A\ast B)=(f\ast A)\ast B, \qquad f\ast (g\ast A)=(f\ast g)\ast A\]
for $f,g\in L^1(\mathbb{R}^{2n})$ and $A,B\in \mathcal{T}^1$.

 For $A \in \mathcal T^1$ we define the \emph{operator involution} by 
\begin{align*}
    A^{\ast_{\textrm{QHA}}} \coloneqq PA^\ast P.
\end{align*}
The involution satisfies 
\begin{equation}\label{eq:involution operator operator}
(A\ast B)^\ast =(A^{\ast_{\textrm{QHA}}})\ast (B^{\ast_{\textrm{QHA}}})
\end{equation}
and 
\begin{equation}\label{eq:involution function operator}
(f\ast A)^{\ast _{\textrm{QHA}}}=(f^{\ast })\ast (A^{\ast_{\textrm{QHA}}})
\end{equation}
for $f\in L^1(\mathbb{R}^{2n})$ and $A,B\in \mathcal{T}^1$.

One can define a commutative product on $L^1 \oplus \mathcal T^1 \coloneqq L^1(\mathbb R^{2n}) \oplus \mathcal T^1(\mathcal H)$ by 
\[(f,A)\ast (g,B)=(f\ast g+A\ast B, f\ast B+g\ast A),\]
where $(f,A),\,(g,B)\in L^1 \oplus \mathcal T^1 $.
Defining the norm on $L^1 \oplus \mathcal T^1$ as
\begin{align*}
    \| (f, A)\| = \| f\|_{L^1} + \| A\|_{\mathcal T^1}
\end{align*}
turns the space into a commutative Banach algebra.
We can also make the space into a Banach $\ast $-algebra by defining the involution as 
\begin{equation*}
    (f, A)^\ast  = (f^\ast ,A^{\ast_{\textrm{QHA}}}).
\end{equation*}
Additionally, the involution is norm isometric:
\[\|(f,A)^\ast \|=\|(f,A)\|.\]
We define the \emph{shift} of $(f,A)\in L^1\oplus \mathcal{T}^1$ by an element $z\in \mathbb{R}^{2n}$ by 
\[\alpha_z(f,A)\coloneqq(\alpha_z f, \alpha_z A).\]
Using the shifts, we define a set $X\subset L^1\oplus \mathcal{T}^1$ to be \emph{shift-invariant} if 
\[\alpha_z(X)\subset X\text{ for all } z\in \mathbb{R}^{2n}.\]

\subsubsection{Fourier Transforms}\label{sec:qha fourier}
Quantum harmonic analysis comes with its own notion of Fourier transforms. We let $\mathcal F_{\sigma}(f)$ be the symplectic Fourier transform, i.e.,
\begin{equation}\label{eq:symplectic}
    \mathcal F_\sigma(f)(z') \coloneqq \frac{1}{(2\pi)^n}\int_{\mathbb R^{2n}}f(z) e^{-i\sigma(z', z)}\,\mathrm{d}z, \quad z' \in \mathbb R^{2n},\ f \in L^{1}(\mathbb{R}^{2n}).
\end{equation}
 Further, we define the \emph{Fourier-Weyl transform} by
\begin{align*}
    \mathcal F_W(A)(z') \coloneqq \tr(AW_{z'}), \quad z' \in \mathbb R^{2n},\ A\in \mathcal T^1.
\end{align*}
We will also occasionally denote the symplectic Fourier transform of $f$ by $\widehat{f} = \mathcal F_\sigma(f)$ and the Fourier-Weyl transform by $\widehat{A}=\mathcal{F}_W(A)$.
The Fourier-Weyl and symplectic Fourier transforms satisfy
\begin{equation}\label{eq:conv_operator_operator}
    \mathcal F_\sigma(A \ast B) = \mathcal F_W(A) \cdot \mathcal F_W(B)
\end{equation}
and 
\begin{equation}\label{eq:conv_function_operator}
    \mathcal{F}_W(f \ast A) = \mathcal{F}_\sigma(f) \cdot \mathcal F_W(A),
\end{equation}
for $A,B\in \mathcal{T}^1$ and $f\in L^1(\mathbb{R}^{2n})$.
Notice that 
\begin{equation}\label{eq:Fourier conjugate}
    \mathcal F_W(A^{\ast_{\textrm{QHA}}}) = \overline{\mathcal F_W(A)}\quad\text{ and }\quad\mathcal F_\sigma(f^{\ast}) = \overline{\mathcal F_\sigma(f)}.
\end{equation}
The map $\mathcal T^1 \ni A \mapsto \mathcal F_W(A)$ is injective and is in fact an isomorphism between $\mathcal{T}^2$ and $L^2(\mathbb{R}^{2n})$. We define the Fourier transform on $L^1\oplus\mathcal{T}^1$ by 
\begin{equation*}
    \mathcal F(f, A) = (\mathcal F_\sigma(f), \mathcal F_W(A)), \quad (f, A) \in L^1 \oplus \mathcal T^1.
\end{equation*}

The inverse of the Fourier-Weyl transform  denoted by $\mathcal F_W^{-1}$ is given by 
\begin{equation*}
    \mathcal F_W^{-1}(f) = \int_{\mathbb R^{2n}} f(z) W_{-z}\,\mathrm{d}z,
\end{equation*}
understood as an integral in strong operator topology for $f \in L^1(\mathbb R^{2n}).$ Using interpolation, one can define the inverse Fourier-Weyl transform on $L^p$ for $1\le p\le 2$. It should be noted that the inverse Fourier-Weyl transform can be generalized to tempered distributions, see \cite{keyl_kiukas_werner16}. In particular, this means that one can define the Fourier-Weyl transform of all compact operators.

The Fourier-Weyl transform satisfies a version of the Riemann Lebesgue lemma and Hausdorff-Young inequality, see \cite{werner84} and \cite[Prop.~6.5 and 6.6]{luef_eirik2018}.
In particular, we have that 
\begin{equation}\label{eq:Riemann Lebesgue}
    \mathcal F(f, A) \in C_0(\mathbb R^{2n}) \oplus C_0(\mathbb R^{2n}).
\end{equation}
 Analogous results are true for $\mathcal F_W^{-1}$, as we show now.
\begin{lemma}\label{inverseRiemannLebesgue}
 Let $1 \leq p \leq 2$ and $\frac{1}{p} + \frac{1}{q} = 1$. For $f \in L^p(\mathbb R^{2n})$ one has $\mathcal F_W^{-1}(f) \in \mathcal T^q(\mathcal H)$ and $\mathcal{F}_W^{-1}$ satisfies the Hausdorff-Young inequality \[\| \mathcal F_W^{-1}(f)\|_{\mathcal T^q} \leq \| f\|_{L^p}.\]
    In particular, for $p=1$ we have the Riemann-Lebesgue lemma: $\mathcal F_W^{-1}(f) \in \mathcal T^\infty $.
\end{lemma}
\begin{proof}
For the proof of the Riemann-Lebesgue lemma, note that $\mathcal F_W^{-1}(f) \in \mathcal T^{2}\subset \mathcal T^{\infty}$ for $f \in L^1(\mathbb{R}^{2n}) \cap L^2(\mathbb{R}^{2n})$. Hence, for arbitrary $f \in L^1(\mathbb{R}^{2n})$ the operator $\mathcal F_W^{-1}(f)$ can be approximated by Hilbert-Schmidt operators in operator norm, hence it is compact.

Notice that for compact operators, the operator norm is given by $\|\cdot\|_{\mathrm{Op}}=\|\cdot\|_{\mathcal{T}^\infty}$.
Since $\| W_{-z}\|_{\mathrm{Op}} = 1$ and $z \mapsto W_{-z}$ is continuous in the strong operator topology, clearly $\| \mathcal F_W^{-1}(f)\|_{\mathrm{Op}} \leq \| f\|_{L^1}$.  Furthermore, we know that $\mathcal F_W^{-1}\colon  L^2(\mathbb R^{2n}) \to \mathcal T^2$ is an isometry.
Hence, the result follows from applying complex interpolation.
\end{proof}

One can also consider the convolution of a complex Radon measure with an operator. Denote by $\operatorname{Meas}(\mathbb R^{2n})$ the set of all complex Radon measures on $\mathbb R^{2n}$, i.e., the set of Borel measures that are finite on compact sets and are both inner and outer regular. Then, for $\mu \in \operatorname{Meas}(\mathbb R^{2n})$ and $A \in \mathcal T^1$, we define
\begin{equation*}
    \mu \ast A \coloneqq \frac{1}{(2\pi)^n}\int_{\mathbb R^{2n}} \alpha_z(A) \,\mathrm{d}\mu(z).
\end{equation*}
The above expression is valid as a Bochner integral in $\mathcal T^1$. In particular, $(2\pi)^n\delta_0 \ast A = A$. One can verify that we still have the product formula
\begin{equation*}
    \mathcal F_W(\mu \ast A) = \mathcal F_\sigma(\mu) \mathcal F_W(A),
\end{equation*}
where
\begin{equation*}
    \mathcal F_\sigma(\mu)(z') \coloneqq \frac{1}{(2\pi)^n}\int_{\mathbb R^{2n}} e^{-i\sigma(z', z)}\, \mathrm{d}\mu(z).
\end{equation*}
Then, if $A \in \mathcal T^1$ is such that $\mathcal F_W(A)(\xi) \neq 0$ for every $\xi \in \mathbb R^{2n}$, we clearly obtain that the map
\begin{equation*}
    \operatorname{Meas}(\mathbb R^{2n}) \ni \mu \mapsto \mu \ast A \in \mathcal T^1
\end{equation*}
is injective.

\subsubsection{Schwartz Operators and Weyl Quantization}
The article \cite{keyl_kiukas_werner16} defines Schwartz and tempered operators, mirroring the Schwartz functions $\mathcal{S}(\mathbb{R}^{2n})$ and tempered distributions $\mathcal{S}'(\mathbb{R}^{2n})$ in harmonic analysis.
We denote by $\mathcal S(\mathcal H) \subset \mathcal T^1(\mathcal H)$ the space of \emph{Schwartz operators}, which is a Frech\'{e}t space that is continuously embedded into $\mathcal T^1(\mathcal H)$. Its topology can be described in terms of a countable family of seminorms, see \cite[Sec.\ 3]{keyl_kiukas_werner16}. We will need the fact that the Fourier-Weyl transform is a topological isomorphism $\mathcal F_W\colon  \mathcal S(\mathcal H) \to \mathcal S(\mathbb R^{2n})$. Hence, a trace class operator $A$ is a Schwartz operator if and only if $\mathcal F_W(A)$ is a Schwartz function. Similarly, a Hilbert-Schmidt operator $A \in \mathcal T^2(\mathcal H)$ is a Schwartz operator if and only if its integral kernel is a Schwartz function.

Having $\mathcal S(\mathcal H)$ available, one can pass to the dual space $\mathcal S'(\mathcal H)$ of \emph{tempered operators}. The bounded operators $\mathcal L(\mathcal H)$ continuously embed into $\mathcal S'(\mathcal H)$ via
\begin{equation*}
    \varphi_A(S) = \tr(AS), \quad S \in \mathcal S(\mathcal H)
\end{equation*}
for $A \in \mathcal L(\mathcal H)$. The Fourier transform, being defined by duality, then extends to a topological isomorphism from $\mathcal S'(\mathcal H)$ to $\mathcal S'(\mathbb R^{2n})$. In particular, we can talk about $\mathcal F_W(A)$ as a tempered distribution on $\mathbb R^{2n}$ for every $A \in \mathcal L(\mathcal H)$. The Fourier-Weyl transform satisfies \eqref{eq:conv_operator_operator} between $\mathcal S(\mathcal H)$ and $\mathcal S'(\mathcal H)$ and \eqref{eq:conv_function_operator} between $\mathcal S(\mathbb R^{2n})$ and $\mathcal S'(\mathcal H)$.

 Define the \emph{Weyl quantization} of an $L^1(\mathbb
{R}^{2n})$ function $f$ to be the operator 
\[A_f\coloneqq P\mathcal{F}_W^{-1}(\mathcal{F}_\sigma(f))P=\mathcal{F}_W^{-1}(\mathcal{F}_\sigma(\widetilde{f})),\]
where $\widetilde{f}(z)=f(-z)$.
Notice that the Weyl quantization also makes sense for $f \in \mathcal{S}'(\mathbb{R}^{2n})$.
Thus, $f \mapsto A_f$ maps from $\mathcal S'(\mathbb R^{2n})$ to $\mathcal S'(\mathcal H)$.
Since the quantization map is also a topological isomorphism from $\mathcal S(\mathbb R^{2n})$ to $\mathcal S(\mathcal H)$, we have
\begin{equation*}
    \mathcal S(\mathcal H) \subset \{ A \in \mathcal T^1(\mathcal H): A = A_f, ~f \in L^1(\mathbb R^{2n})\}.
\end{equation*}
Taking the Weyl quantization of the delta function \[(2\pi)^n\delta_{0}\in \mathcal{S}'(\mathbb R^{2n})\] gives $A_{(2\pi)^n\delta_{0}}=2^n P$, where $P$ is the parity operator.
Note that the quantization map satisfies $\alpha_z(A_f) = A_{\alpha_{-z}(f)}$. For $g \in L^p(\mathbb R^{2n})$ where $1 \leq p \leq \infty$ and $f \in L^1(\mathbb R^{2n})$ we have the convolution formulas
\begin{equation}\label{eq: quantization of convolution}
    f \ast A_g = A_{\widetilde{f}\ast g}, \quad A_f \ast A_g = \widetilde{f \ast g}.
\end{equation}
 This gives us the following quantization identity
\[2^nf\ast P=A_{\widetilde{f}},\qquad 2^nA_f\ast P=\widetilde{f}\]
for $f\in \mathcal{S}(\mathbb{R}^{2n})$.

\subsection{Modulation Invariance}
In this section, we will develop a notion of modulation for operators. This allows us to define modulation-invariant spaces of operators. Motivated by the fact that the classical Fourier transform takes shifts to modulations and vice versa, we have the following definition.

\begin{definition}
We say that a subspace $\mathcal A \subset \mathcal T^{1}$ is \emph{modulation-invariant} if \[\mathcal F_W(\mathcal A) \coloneqq \left\{\mathcal{F}_{W}(A) :  A \in \mathcal{A}\right\}\] is a shift-invariant subspace of $C_{0}(\mathbb{R}^{2n})$.
\end{definition}

\begin{lemma} \label{lem:modulation_invariance}
Let $\mathcal A \subset \mathcal T^1$ be a subspace. Then $\mathcal A$ is modulation-invariant if and only if $W_{z} A W_{z} \in \mathcal A$ for every $A \in \mathcal A$ and $z \in \mathbb R^{2n}$.
\end{lemma}
\begin{proof}
For $z' \in \mathbb{R}^{2n}$ the claim follows easily from the computation
\begin{align*}
    \alpha_{z}\left(\mathcal{F}_{W}(A)(z')\right) & = \tr(W_{z' - z}A) \\  
    & = \tr\left( W_{-\frac{z}{2}}AW_{-\frac{z}{2}} W_{z'}\right) \\ & = \mathcal{F}_{W}\left(W_{-\frac{z}{2}}AW_{-\frac{z}{2}}\right)(z').\qedhere
\end{align*}
\end{proof}

Motivated by the proof of Lemma~\ref{lem:modulation_invariance} we give the following definition for the modulation of an operator.
\begin{definition}
The \emph{modulation} of an operator $A \in \mathcal{T}^{1}$ by $z \in \mathbb{R}^{2n}$ is defined by \[\gamma_{z}(A)\coloneqq W_{-z/2}AW_{-z/2}.\]
For $f \in L^{1}(\mathbb{R}^{2n})$ we define the modulation by
\begin{equation*}
    \gamma_{z}(f)(z') \coloneqq e^{-i\sigma(z', z)} f(z').
\end{equation*}
\end{definition}

It is straightforward to verify the formulas 
\begin{align*}
    \gamma_{z}(f \ast g) & = (\gamma_{z}f) \ast (\gamma_{z} g), \\
    \gamma_{z}(f \ast A) & = (\gamma_{z} f) \ast (\gamma_{z}A), \\
    \gamma_{z}(A \ast B) & = (\gamma_{z} A) \ast (\gamma_{z}B),
\end{align*}
for $f, g \in L^{1}(\mathbb{R}^{2n})$ and $A, B \in \mathcal T^1$.
Notice that for $A \in \mathcal{T}^{1}$ and $f\in L^1(\mathbb{R}^{2n})$ we can write 
\[\mathcal{F}_{W}(A)(z) = \tr(\gamma_{-z}A)\quad \text{ and } \quad \mathcal{F}_{\sigma}(f)(z) = \int_{\mathbb{R}^{2n}}\gamma_{-z}(f)(z')\, \mathrm{d} z'.\]

\begin{lemma}
Let $\mathcal A$ be a shift-invariant subspace of $\mathcal T^1$. Then the following are equivalent:
\begin{enumerate}[1)]
    \item\label{eq:modulation_1} $\mathcal A$ is modulation-invariant,
    \item $W_{z} A \in \mathcal A$ for every $A\in \mathcal A$ and $z \in \mathbb R^{2n}$,
    \item $AW_{z} \in \mathcal A$ for every $A \in \mathcal A$ and $z\in \mathbb R^{2n}$,
    \item\label{eq:modulation_4} $W_{z'}AW_{z} \in \mathcal A$ for every $A \in \mathcal A$ and $z',z\in \mathbb R^{2n}$,
    \item The space consisting of all integral kernels of $A\in \mathcal{A}$ is shift- and modulation-invariant.
\end{enumerate} 
\end{lemma}
\begin{proof}
The equivalence of \ref{eq:modulation_1}--\ref{eq:modulation_4}\ follows from shift-invariance together with the observation
\begin{align*}
\gamma_{z}\alpha_{z'}(\widehat{A}) 
&=\mathcal{F}_W(\alpha_{z}\gamma_{z'}A)\\
&=\mathcal{F}_W(W_{z}W_{-z'/2}AW_{-z'/2}W_{z})\\
&=e^{i\sigma(z',z)/2}\mathcal{F}_W(W_{z-z'/2}AW_{-z'/2-z}).
\end{align*}

For the last equivalence, recall that a trace class operator can be written as 
\[A\phi(s)=\int_{\mathbb{R}^n}K_A(s,t)\phi(t)\,\mathrm{d}t,\quad \phi \in \mathcal{H}.\]
The integral kernel of $W_{z'}AW_z$ with $z = (x, \xi)$ and $z' = (x', \xi')$ is given by
\begin{equation*}
(s, t) \mapsto e^{i\xi' \cdot s - \frac{i}{2} \xi'\cdot x'}e^{i\xi\cdot t + \frac{i}{2} \xi\cdot x}K_A(s-x', t+x)=e^{ \frac{i}{2} \xi\cdot x - \frac{i}{2} \xi'\cdot x'}\gamma_{(\xi,-\xi')}\alpha_{(x',-x)}K_A(s, t).\qedhere
\end{equation*}
\end{proof}
Analogously to the case of the shift action $\alpha_x$, one also proves the following facts:
\begin{lemma}\label{lemma:modulation_cont_t1}
    $\gamma_x$ acts strongly continuously and isometrically on the Schatten classes $\mathcal T^p(\mathcal H)$ for every $1 \leq p \leq \infty$. It acts continuous in weak$^\ast$ topology on $\mathcal L(\mathcal H)$.
\end{lemma}
\begin{proof}
    The proof is entirely analogous as for the action $\alpha_x$: Verify the statements first for rank one operators, using that $W_x$ is a strongly continuous projective unitary representation. Then, obtain the result on all of $\mathcal T^p(\mathcal H)$ by approximation through finite rank operators. 
\end{proof}

\section{Structure Theory of \texorpdfstring{$L^1 \oplus \mathcal T^1$}{}}\label{sec: structure theory}
In Section \ref{sec: QHA convolutions} we defined a product structure for $L^1 \oplus \mathcal T^1$, making it into an involutive commutative Banach algebra. As such we can study the closed ideals of this algebra. We will start with the study of the regular maximal ideals, i.e., the Gelfand theory, and thereafter study the closed graded ideals of $L^1 \oplus \mathcal T^1$. 

\subsection{Gelfand Theory for \texorpdfstring{$L^1 \oplus \mathcal T^1$}{}}\label{sec: gelfand}
 In this section we will study the Gelfand theory of $L^1 \oplus \mathcal T^1$. We say that an ideal $I$ is \emph{regular} if the quotient $L^1 \oplus \mathcal T^1/I$ is unital, see \cite[Sec.~D4]{rudin}. It is well known that the regular maximal ideals are given by the kernel of the multiplicative linear functionals.

\begin{proposition}\label{prop: mult char}
     The set of nonzero multiplicative linear functionals on $L^1\oplus \mathcal{T}^1$ is para\-me\-trized by $\mathbb R^{2n} \times \mathbb Z_2$ as
\begin{equation*}
    \chi_{z, j}(f, A) = \begin{cases}
        \widehat{f}(z) + \widehat{A}(z), &\quad j = 0,\\
        \widehat{f}(z) - \widehat{A}(z), &\quad j = 1.
    \end{cases}
\end{equation*}
We write the set of all nonzero multiplicative linear functionals, the \emph{Gelfand spectrum}, as $\chi_{\mathbb R^{2n} \times \mathbb Z_2}$. The Gelfand spectrum topology of $\chi_{\mathbb R^{2n} \times \mathbb Z_2}$ agrees with the standard product topology of the index set $\mathbb{R}^{2n}\times \mathbb{Z}_2$.
\end{proposition}
\begin{proof}
    
Let $\chi\colon  L^1 \oplus \mathcal T^1 \to \mathbb C$ be a nonzero multiplicative linear functional. Note that by using linearity, $\chi$ is determined by its actions on functions and operators separately. By the Gelfand theory of $L^1(\mathbb{R}^{2n})$, the multiplicative character $\chi$ acts on $L^1\oplus\{0\}$ by 
\begin{equation*}
    \chi(f, 0) = \widehat{f}(z) \quad \text{ for some }z \in \mathbb R^{2n}.
\end{equation*}
For $A, B \in \mathcal T^1$ we have 
\begin{align*}
    \chi(0, A) \cdot \chi(0, B) = \chi((0, A) \ast (0, B)) = \chi( A \ast B, 0) = \mathcal F_\sigma(A \ast B)(z) = \widehat{A}(z) \cdot \widehat{B}(z).
\end{align*}
Letting $A = B$, we see that
\begin{align*}
    \chi(0, A) = \pm \widehat{A}(z).
\end{align*}
Hence, we obtain that
\begin{align*}
    \chi(f, A) = \widehat{f}(z) \pm \widehat{A}(z).
\end{align*}
As one easily verifies, every such $\chi$ is a multiplicative linear functional. 

Let us now prove the equivalence of the topologies.
 A net $(z_\gamma) \subset \mathbb R^{2n}$ converges to $z \in \mathbb R^{2n}$ if and only if $\chi_{z_\gamma, j} \to \chi_{z,j}$ in the Gelfand spectrum topology. Further, it can never happen that $\chi_{z_\gamma, j} \to \chi_{z, j+1}$. If so, we would have $\widehat{A}(z) = -\widehat{A}(z)$, implying that $\widehat{A}(z) = 0$ for every $A \in \mathcal T^1$. This can not happen, since there exist operators with nonvanishing Fourier-Weyl transform, e.g., $\phi \otimes \phi$ where $\phi(x)=e^{-x^2}$.
\end{proof}
As a corollary, the maximal ideal space of $L^1\oplus \mathcal{T}^1$ consists of two disjoint copies of $\mathbb R^{2n}$. 

\begin{remark}
Recall by \eqref{eq:Riemann Lebesgue} we have \[\mathcal{F}(L^1 \oplus \mathcal T^1)\subset C_0(\mathbb R^{2n}) \oplus C_0(\mathbb R^{2n}).\] 
Using the convolution theorem, we can view the Fourier transform as a Banach algebra homomorphism
\begin{equation*}
    \mathcal F\colon  L^1 \oplus \mathcal T^1 \to \ell^1(\mathbb Z_2, C_0(\mathbb R^{2n})).
\end{equation*}
\end{remark}
Denote by 
\[\Gamma\colon L^1\oplus\mathcal{T}^1\to (\chi_{\mathbb R^{2n} \times \mathbb Z_2}\to \mathbb{C}),\]
the \emph{Gelfand representation} defined by
\[\Gamma(f,A)\chi_{z,j}\coloneqq \chi_{z,j}(f,A)=\widehat{f}(z)+(-1)^j\widehat{A}(z).\]
We will refer to the map $\Gamma(f,A)$ is the \emph{Gelfand transform} of $(f,A)$.
\begin{remark}\label{re:denseness}
    The Gelfand representation can be viewed as a map $\Gamma\colon L^1\otimes\mathcal{T}^1\to C_0(\mathbb{R}^{2n}\times \mathbb{Z}_2)$
    given by \[\Gamma(f,A)(z,j)=\widehat{f}(z)+(-1)^j\widehat{A}(z).\]
    Since both Fourier transforms have dense range, it follows that the image of the Gelfand representation is dense in $C_0(\mathbb{R}^{2n}\times \mathbb{Z}_2)$.
\end{remark}
Recall that a ring is called \emph{Jacobson semisimple} (or \emph{semiprimitive}) if the Jacobson ideal is zero, i.e., the intersection of all the maximal ideals is zero. A property of the Gelfand transform is that $\ker(\Gamma)$ is precisely the Jacobson ideal. It follows from the Gelfand-Naimark theorem that any $C^{\ast }$-algebra is Jacobson semisimple. The following result shows that $L^1 \oplus \mathcal T^1$ shares this property.
\begin{corollary}\label{cor: Jacobson semisimple}
   The algebra  $L^1 \oplus \mathcal T^1$ is Jacobson semisimple.
\end{corollary}
\begin{proof}
Let $(f, A)$ be such that $\Gamma(f, A) = 0$. Then for every $z \in \mathbb R^{2n}$ we obtain 
\begin{align*}
    \widehat{f}(z) + \widehat{A}(z) = 0 = \widehat{f}(z) - \widehat{A}(z).
\end{align*}
Hence $\widehat{f}=\widehat{A} = 0  $. Since both the symplectic Fourier transform and the Fourier-Weyl transform are injective, this yields that $(f, A) = 0$. Thus, the Jacobson radical is trivial. 
\end{proof}

We say that an involutive Banach algebra 
$X$ is \emph{symmetric} if its Gelfand representation satisfies 
\[\Gamma(a^\ast )=\overline{\Gamma(a)} \quad \text{ for all } a\in X.\]
Using \eqref{eq:Fourier conjugate} we see that $L^1\oplus \mathcal{T}^1$ is symmetric since
\begin{align*}
    \Gamma((f, A)^\ast)(z, j) = \overline{\mathcal F_\sigma(f)(z)} + (-1)^j \overline{\mathcal F_W(A)(z)} = \overline{\Gamma(f, A)(z, j)}.
\end{align*}
Using that $L^1\oplus \mathcal{T}^1$ is symmetric and \cite[Prop.\ 1.14c.]{folland2016course}, we get another proof of the denseness result in Remark \ref{re:denseness}.

\subsection{Closed Ideals of \texorpdfstring{$L^1 \oplus \mathcal T^1$}{}}\label{sec:closed_ideals}
Having discussed Gelfand theory, we now turn to the task of understanding closed ideals of $L^1 \oplus \mathcal T^1$ in general. 
Recall that the closed ideals of $L^1(\mathbb{R}^{2n})$ are precisely the closed shift-invariant subspaces of $L^1(\mathbb{R}^{2n})$, see e.g., \cite[Thm.~2.45]{folland2016course}. Having a well-defined shift operator on $L^1\oplus \mathcal{T}^1$, one might hope that carries over to closed ideals in this setting. However, the subspace $\{0\}\oplus \mathcal{T}^1$ is closed and shift-invariant, but not an ideal. 
\begin{definition}
We say that a subspace $M\subset L^1\oplus \mathcal{T}^1$ is an \emph{$L^1$ module} if for all $(f,A)\in M$ and $g\in L^1(\mathbb{R}^{2n})$ we have 
\[g\ast (f,A)\coloneqq (g,0)\ast (f,A)=(g\ast f,g\ast A)\in M.\]
\end{definition}
Notice that all ideals of $L^1\oplus \mathcal{T}^1$ are $L^1$ modules.
\begin{proposition}\label{prop:L^1 module}
Let $M \subset L^1 \oplus \mathcal T^1$ be a closed subspace. Then $M$ is shift-invariant if and only if $M$ is an $L^1$ module.
\end{proposition}
\begin{proof}
If $M$ is shift-invariant, then for every $(f, A) \in M$ and $g \in L^1$ the convolution satisfies
\begin{align*}
    g \ast (f, A) = \int_{\mathbb{R}^{2n}} g(z) \alpha_z(f,A)\,\mathrm{d}z\in M.
\end{align*}
On the other hand, if $M$ is an $L^1$ module, then letting $\{g_t\}_{t>0}$ be a normalized approximate identity of $L^1(\mathbb{R}^{2n})$ gives
\begin{equation*}
    \alpha_z(f, A) = \lim_{t \to 0} g_t \ast \alpha_z(f, A) = \lim_{t \to 0} \alpha_z(g_t) \ast (f, A) \in M. \qedhere 
\end{equation*}
\end{proof}
As a consequence, for $I$ being a closed ideal of $L^1 \oplus \mathcal T^1$, it is necessary to be shift-invariant. We continue our discussion with the following class of well-behaved ideals:
\begin{definition}
An ideal $I \subset L^1 \oplus \mathcal T^1$ is said to be \emph{graded} if $I = I_{L^1} \oplus I_{\mathcal{T}^1}$, where $I_{L^1} = I \cap (L^1(\mathbb{R}^{2n})\oplus\{0\})$ and $I_{\mathcal{T}^1} = I \cap (\{0\}\oplus \mathcal T^1)$.  
\end{definition}
We note that $I_{L^1}$ and $I_{\mathcal T^1}$ in the above definition are \emph{corresponding spaces} in the sense of Werner \cite{werner84}.

 Note that a closed ideal $I$ is a graded ideal if and only if $(f,0) \in I$ whenever $(f, A)\in I$. It follows that the regular maximal ideals 
 \begin{align*}
    I_{(z, j)} \coloneqq \{ (f, A) \in L^1 \oplus \mathcal T^1: \widehat{f}(z)  + (-1)^j \widehat{A}(z) = 0\}
\end{align*}
are never graded.
\begin{lemma}\label{lem:graded_closed_ideals}
A closed subspace $I \subset L^1 \oplus \mathcal T^1$ is a graded ideal if and only if $I = I_{L^1} \oplus \overline{(\mathcal T^1 \ast I_{L^1})}$ for some closed ideal $I_{L^1}$ in $L^1(\mathbb{R}^{2n})$.
\end{lemma}
\begin{proof}
Let us first quickly check that $I = I_{L^1} \oplus \overline{(\mathcal T^1 \ast I_{L^1})}$ is indeed a closed ideal. If $(f, A \ast g) \in I$ and $(h, B) \in L^1 \oplus \mathcal T^1$, then
\begin{align*}
    (h, B) \ast (f, A \ast g) = (h \ast f + B \ast A \ast g, A \ast g\ast  h  + f \ast B).
\end{align*}
Since $f, g \in I_{L^1}$ and $I_{L^1}$ is an ideal, we have $h \ast f \in I_{L^1}$ and $B \ast A \ast g \in I_{L^1}$, where we have used that $B \ast A \in L^1(\mathbb{R}^{2n})$. Further, $A \ast g \ast h$ and $f \ast B$ are clearly contained in $\mathcal T^1 \ast I_{L^1}$, making $I$ into an ideal. The fact that $I$ is closed is inherited from $I_{L^1}$ being closed. 

On the other hand, assume that $I= I_{L^1} \oplus I_{\mathcal{T}^1}$ is a graded closed ideal of $L^1 \oplus \mathcal T^1$. 
Note that $I_{L^1}$ is a closed ideal of $L^1(\mathbb{R}^{2n})$. Using the multiplication, we get $(\{0\}\oplus \mathcal T^1) \ast I_{L^1} \subset I_{\mathcal{T}^1}$ and $(\{0\}\oplus \mathcal T^1) \ast I_{\mathcal{T}^1} \subset I_{L^1}$.  Further, as $I$ is an $L^1$ module by Proposition \ref{prop:L^1 module}, both $I_{L^1}$ and $I_{\mathcal{T}^1}$ are shift-invariant. Thus using \cite[Thm.~4.1]{werner84} gives that $I_{\mathcal{T}^1} = \overline{\mathcal T^1 \ast I_{L^1}}$, finishing the proof.
\end{proof}
\begin{remark}\label{re graded}
\begin{enumerate}[(1)]
\item For $f \in L^1(\mathbb R^{2n})$ we let $Z(f)$ denote the closed set
\begin{align*}
    Z(f) \coloneqq \{ z \in \mathbb R^{2n}: \widehat{f}(z) = 0\}.
\end{align*}
Then, for a closed ideal $I_{L^1}$ of $L^1(\mathbb R^{2n})$, we set \[Z(I_{L^1}) \coloneqq \bigcap_{f \in I_{L^1}} Z(f),\] which is closed. Malliavin's theorem \cite[Sec.~7.6]{rudin} states that spectral synthesis fails for $L^1(\mathbb R^{2n})$, that is: It is not true that every closed ideal $I_{L^1}$ is uniquely determined by $Z(I_{L^1})$. 

By analogy, for $(f, A) \in L^1 \oplus \mathcal T^1$ we let
\begin{align*}
    Z(f,A) \coloneqq \{ (z, j) \in \mathbb R^{2n} \times \mathbb Z_2: \widehat{f}(z) + (-1)^j \widehat{A}(z) = 0\}.
\end{align*}
If $I$ is a closed ideal of $L^1 \oplus \mathcal T^1$ we denote by
\[Z(I) \coloneqq \bigcap_{(f,A)\in I} Z(f,A).\] 
When $I$ is a graded, by Lemma \ref{lem:graded_closed_ideals} we get that
    $Z(I) = Z(I_{L^1} ) \times \mathbb Z_2$.
This means that for graded ideals the spectral synthesis again fails, i.e., $I $ is not uniquely determined by this set $Z(I)$.
\item\label{re maximal} Every proper closed graded ideal $I=I_{L^1} \oplus \overline{(\mathcal T^1 \ast I_{L^1})}$ is contained in a regular maximal ideal. If $Z(I) \neq \emptyset$, we have that for $(z, j) \in Z(I)$ the inclusion $I \subset I_{(z, j)}$ holds. In the case $Z(I) = \emptyset$, it remains to show that $I=L^1 \oplus \mathcal T^1$. When $Z(I) = \emptyset$, we have $Z(I_{L^1})=\emptyset$. By the Tauberian theorem on $L^1(\mathbb{R}^{2n})$, we have that $I_{L^1}=L^1(\mathbb{R}^{2n})$. The statement follows from the fact that $\overline{\mathcal{T}^1\ast L^1(\mathbb{R}^{2n})}=\mathcal{T}^1$.
\end{enumerate}
\end{remark}

Every closed ideal contains a maximal graded ideal. To illustrate this, let us consider the regular maximal ideals.
\begin{example}
For $(z, j) \in \mathbb R^{2n} \times \mathbb Z_2$, we have
\begin{align*}
    I_0 &= I_{(z, j)} \cap ( L^1(\mathbb{R}^{2n})\oplus \{0\} )= \{ f\in L^1(\mathbb{R}^{2n}): \widehat{f}(z) = 0\}\\
     I_1 &= I_{(z, j)} \cap  ( \{0\}\oplus \mathcal T^1 ) = \{ A \in \mathcal T^1: ~\widehat{A}(z) = 0\}.
\end{align*}
In this case, $I_0 \oplus I_1$ is a graded closed ideal given by
\begin{align*}
    I_0 \oplus I_1 = I_{(z, 0)} \cap I_{(z, 1)}.
\end{align*}
\end{example}
The construction of graded subideals works more generally. Denote by $J$ the map $J(f,A) \coloneqq (f, -A)$.
\begin{lemma}\label{lem: idempotent J}
Let $I$ be a closed ideal of $L^1 \oplus  \mathcal T^1$. Then
\begin{enumerate}[1)]
    \item $J(I)$ is a closed ideal of $L^1 \oplus \mathcal T^1$.
    \item $I$ is graded if and only if $J(I) = I$.
    \item\label{lem: idempotent J graded sub} $I \cap J(I)$ is a closed graded ideal.
    \item $Z(J(I)) = \{ (z, j+1) \in \mathbb R^{2n} \times \mathbb Z_2: ~(z, j) \in Z(I)\}$.
    \item $Z(I \cap J(I)) = Z(I) \cup Z(J(I)) = \{ z \in \mathbb R^{2n}: ~(z, 0) \in Z(I) \text{ or } (z, 1) \in Z(I)\} \times \mathbb Z_2$.
\end{enumerate}
\end{lemma}
The proof of this lemma is straightforward and left to the reader.
Note that the graded ideal $I \cap J(I)$ can be trivial, as the following example shows:
\begin{example}
We consider only $n = 1$, but the example can analogously be carried out for $n > 1$. We set
\begin{align*}
    \Omega_j = \{ (x, y) \in \mathbb R^2: ~(-1)^jy \geq 0\}, \quad j\in \mathbb{Z}_2.
\end{align*}
Define the closed ideal
\begin{align*}
    I \coloneqq \{ (f, A) \in L^1 \oplus \mathcal T^1: ~\widehat{f}(z) + (-1)^j \widehat{A}(z) = 0 \text{ for } z \in \Omega_j,\, j\in \mathbb{Z}_2 \}.
\end{align*}
By the definition, we have $Z(I) = (\Omega_0 \times \{ 0\}) \cup (\Omega_1 \times \{ 1\})$. 
Thus, we clearly have
\begin{align*}
    Z(I \cap J(I))=Z(\{(0,0)\}) = \mathbb R^2 \times \mathbb Z_2.
\end{align*}
\end{example}
\begin{proposition}\label{prop: zeros}
Let $I$ be a closed ideal of $L^1 \oplus\mathcal T^1$. If $Z(I) = \emptyset$, then $I = L^1 \oplus \mathcal T^1$.
\end{proposition}
\begin{proof}
We clearly have $I \cap J(I) \subset I$. By Lemma \ref{lem: idempotent J} \ref{lem: idempotent J graded sub}, $I \cap J(I)$ is a closed graded ideal with $Z(I \cap J(I)) = \emptyset$. Thus, $I \cap J(I) = L^1 \oplus \mathcal T^1$ by Remark \ref{re graded} \ref{re maximal}.
\end{proof}
\begin{corollary}
Every proper closed ideal of $L^1 \oplus \mathcal T^1$ is contained in a regular maximal ideal.
\end{corollary}
\begin{proof}
By Proposition \ref{prop: zeros}, if $I$ is proper $Z(I) \neq \emptyset$. Hence, for $(z, j) \in Z(I)$, we have $I \subset I_{(z, j)}$.
\end{proof}

\begin{corollary}
Every maximal closed ideal of $L^1 \oplus \mathcal T^1$ is regular.
\end{corollary}
\begin{proof}
If $I$ is a (proper) maximal closed ideal, then it is contained in a regular maximal ideal by the previous result. By maximality, both ideals have to agree.
\end{proof}
Having looked at the maximal closed graded subideal of $I$, we now turn to the minimal closed graded ideal containing $I$.
\begin{lemma}
Let $I \subset L^1 \oplus \mathcal T^1$ be a closed ideal.
\begin{enumerate}
    \item $\overline{I + J(I)}$ is a graded ideal.
    \item $Z(\overline{I + J(I)} )= Z(I) \cap Z(J(I)) = \{ z \in \mathbb R^{2n}: ~(z, 0) \in Z(I) \text{ and } (z, 1) \in Z(I)\} \times \mathbb Z_2$.
\end{enumerate}
\end{lemma}
Again, we leave the proof of the lemma to the reader.
\begin{example}
In the case of the maximal ideal $ I_{(z, j)}$, we have that \[\overline{I_{(z, j)} + J(I_{(z, j)})}=L^1\oplus \mathcal{T}^1.\]
This is a simple consequence of Proposition \ref{prop: zeros}
together with $J(I_{(z, j)})=I_{(z, j+1)}$.
\end{example}
Let us now move on to prove some general properties of closed ideals. Let $V\subset L^1\oplus \mathcal{T}^1$, then 
\[V'\coloneqq\overline{(\{0\}\oplus \mathcal T^1 ) \ast V}.\]
\begin{proposition}
Let $V \subset L^1 \oplus \mathcal T^1$ a closed, shift-invariant subspace.
\begin{enumerate}
    \item $V'$ is a closed, shift-invariant subspace.
    \item $Z(V') = Z(V)$.
    \item $V = V'' $.
    \item $\overline{V + V'}$ is a closed ideal of $L^1 \oplus \mathcal T^1$.
    \item $Z(\overline{V + V'}) = Z(V)$.
\end{enumerate}
In particular, if $I \subset L^1 \oplus \mathcal T^1$ is a closed ideal then $I = \overline{I + I'}$.
\end{proposition}
\begin{proof}
\begin{enumerate}
    \item By definition $V'$ is closed. Further, $(\{0\}\oplus\mathcal T^1 ) \ast V$ is shift-invariant, by the definition of convolution. Taking the closure preserves shift-invariance.
    \item Let $B$ be a regular operator, i.e., $\mathcal{F}_W(B)(z)$ is nonzero for all $z\in \mathbb{R}^{2n}$.
    Then by 
    \eqref{eq:conv_operator_operator} and \eqref{eq:conv_function_operator} we have that
    \begin{align*}
        \widehat{A \ast B}(z) + (-1)^j \widehat{f \ast B}(z) =- \widehat{B}(z) (\widehat{f}(z) + (-1)^{j}\widehat{A}(z)), \quad (f,A)\in V.
    \end{align*}
    Hence the result follows from the definition of $Z(V')$.
    \item This follows from the fact that $\mathcal T^1 \ast \mathcal T^1 \subset L^1$ is dense, the assumptions on $V$ and associativity of the convolutions.
    \item Clearly, $\overline{V + V'}$ is an $L^1$ module. The definition of $V'$ makes $V'$ also invariant under convolution by $\mathcal T^1$.
    \item The final property is an immediate consequence of 2.\qedhere
\end{enumerate}
\end{proof}
The above result tells us how to construct ideals of $L^1 \oplus \mathcal T^1$: Pick any closed, shift-invariant subspace $V$ of $L^1 \oplus \mathcal T^1$ and form $\overline{V + V'}$. Then, if $Z(V) \neq \emptyset$, we are guaranteed to obtain a proper closed ideal of $L^1 \oplus \mathcal T^1$. Further, any closed ideal is of this form. Finally, if we let $V \subset L^1(\mathbb{R}^{2n}) \oplus \{ 0\}$ be a closed shift-invariant subspace we obtain a graded ideal. Additionally, by Lemma \ref{lem:graded_closed_ideals} any 
closed graded ideal is of this form.

\section{Quantum Segal Algebras}\label{sec: Quantum Segal}
Moving on from the closed ideals and $L^1$ modules, we will in this section study dense $L^1$ modules endowed with a shift invariant norm.
In doing this, we will define a version of Segal algebras in $L^1\oplus \mathcal{T}^1$. 

\begin{definition}
A \emph{quantum Segal algebra (QSA)} is a pair $(QS, \| \cdot\|_{QS})$ where:
\begin{enumerate}[(QS1)]
\item \label{label:QS1} $QS$ is a dense subspace of $L^1 \oplus \mathcal T^1$.
\item \label{label:QS2} The space $(QS, \|\cdot\|_{QS})$ is a Banach algebra with multiplication inherited from $L^1 \oplus \mathcal T^1$. 
\item \label{label:QS3} The space $QS$ is shift-invariant.
\item \label{label:QS4} The shifts $\alpha_z$ are norm-isometric and continuous on $(QS, \| \cdot\|_{QS})$ for every $z\in \mathbb{R}^{2n}$.
\end{enumerate}
If, additionally, the involution satisfies $QS^\ast =QS$ and  $\|(f,A)^\ast\|_{QS}=\|(f,A)\|_{QS}$ for all $(f,A)\in QS$, then we refer to $(QS, \| \cdot\|_{QS})$ as \emph{star-symmetric}.
\end{definition}

As a first consequence of the definition, we obtain that every quantum Segal algebra continuously embeds into the ambient space:
\begin{proposition}
Let $(QS, \| \cdot\|_{QS})$ be a quantum Segal algebra. Then, there is a constant $C > 0$ such that for each $(f, A) \in QS$ we have
\begin{equation*}
    \| (f, A)\|_{L^1 \oplus \mathcal T^1} \leq C\| (f, A)\|_{QS}.
\end{equation*}
\end{proposition}

The proposition is a direct application of the following well-known result on continuity of Banach algebra homomorphisms.
Recall by Corollary \ref{cor: Jacobson semisimple} that $L^1 \oplus \mathcal T^1$ is Jacobson semisimple.
\begin{lemma}\label{lem:homeo_implies_cont}
Let $\mathcal A$ and $\mathcal B$ be commutative Banach algebras, where $\mathcal B$ is further assumed to be Jacobson semisimple. Then any algebra homomorphism $\varphi\colon  \mathcal A \to \mathcal B$  is continuous.
\end{lemma}
Since the proof is very short, we present it for the reader's convenience.
\begin{proof}
Let $(x_n)_{n \in \mathbb N}$ be a sequence in $\mathcal A$ converging to $0 \in \mathcal A$ and $y=\lim_{n\to \infty}\varphi(x_n) $. Denote by $\mathcal M(\mathcal A)$ and $\mathcal M(\mathcal B)$ the space of multiplicative linear functionals on $\mathcal A$ and $\mathcal B$, respectively. Then, for every $\chi \in \mathcal M(\mathcal B)$ we have $\chi\circ \varphi \in \mathcal M(\mathcal A)$, hence $\chi\circ \varphi$ is continuous. Thus,
\begin{align*}
    \chi(y) = \chi(\lim_{n \to \infty} \varphi(x_n)) = \lim_{n \to \infty} \chi(\varphi(x_n)) = \chi(\varphi(0)) = 0.
\end{align*}
Hence, $\chi(y) = 0$ for every $\psi \in \mathcal M(\mathcal B)$. Since $\mathcal B$ is assumed to be Jacobson semisimple, it is $y = 0$. Thus, by the closed graph theorem, $\varphi$ is continuous.
\end{proof}
Recall that every Segal algebra is an $L^1$ module by  \eqref{eq:banach_ideal_bound}.
Let $\{g_t\}_{t>0}$ be a normalized approximate identity for $L^1(\mathbb{R}^{2n})$, then $(g_t,0)$ is an approximate identity for $L^1\oplus \mathcal{T}^1$. 
 Hence, every quantum Segal algebra $QS$ is an \emph{essential} Banach module, and the Cohen-Hewitt factorization theorem implies that
\begin{align*}
    QS = \{ g \ast (f, A): ~g \in L^1(\mathbb R^{2n}), \, (f, A) \in QS\}.
\end{align*}
As in the case for Segal algebras, see, e.g.,  \cite{dunford1974}, the property of being an essential $L^1$ module uniquely determines quantum Segal algebras among dense subalgebras of $L^1 \oplus \mathcal T^1$ in the following sense.

\begin{proposition}
Let $\mathcal A \subset L^1 \oplus \mathcal T^1$ be a dense subalgebra, such that $(\mathcal A, \| \cdot\|_{\mathcal A})$ is a Banach algebra. Then $\mathcal A$ is an essential $L^1$ Banach module if and only if it is a quantum Segal algebra.
\end{proposition}
\begin{proof}
Since $\mathcal{A}$ is essential, we have that every element of $(f, A)=h\ast(g,B)$, for $h \in L^1(\mathbb R^{2n})$ and $(g, B) \in \mathcal A$.
 Hence $\mathcal A$ is shift-invariant since
\begin{align*}
    \alpha_z(f, A) = \alpha_z(h \ast (g, B)) = \alpha_z(h) \ast (g, B) \in \mathcal A \text{ for all } z\in \mathbb{R}^{2n}.
\end{align*}
Furthermore, for $z \to 0$ we have 
\begin{align*}
    \| \alpha_z(f, A) - (f, A)\|_{\mathcal A} = \| \alpha_z(h) \ast (g, B) - h \ast (g,B)\|_{\mathcal A} \leq \| \alpha_z(h) - h\|_{L^1} \| (g, B)\|_{\mathcal A} \to 0.
\end{align*}
The only thing left is showing that the shifts act isometric with respect to $\|\cdot\|_{\mathcal{A}}$. On the one hand, for $\{g_{t}\}_{t > 0} \subset L^1(\mathbb R^{2n})$ a normalized approximate identity and $(f, A) \in \mathcal A$ we have
\begin{align*}
    \| \alpha_z(f, A)\|_{\mathcal A} = \lim_{t \to 0} \| \alpha_z(g_t) \ast (f, A)\|_{\mathcal A} \leq \limsup_{t \to 0} \| \alpha_z(g_t)\|_{L^1} \| (f, A)\|_{\mathcal A} = \| (f, A)\|_{\mathcal A}.
\end{align*}
The other direction follows from applying the same argument to $ \alpha_{-z}(\alpha_{z}(f, A)) = (f, A)$.
\end{proof}

We have seen that quantum Segal algebras are $L^1$ modules, however, they are not in general $L^1 \oplus \mathcal T^1$ ideals. Remark \ref{remex} below gives an example for this. In particular, this means that quantum Segal algebras are not abstract Segal algebras in the sense of Burnham \cite{Burnham1972}. Hence we should not hope for the full ideal theorem that Segal algebras satisfy. Nevertheless, there is a weak version:
\begin{proposition}
Let $QS$ be a quantum Segal algebra.
\begin{enumerate}[(1)]
    \item For every closed ideal $I$ of $L^1 \oplus \mathcal T^1$, the set $I \cap QS$ is a closed ideal of $QS$.
    \item If $I$ is a closed ideal of $QS$, then the $L^1 \oplus \mathcal T^1$-closure of $I$ is a closed ideal of $L^1 \oplus \mathcal T^1$.  Further, $I \subseteq \overline{I} \cap QS$.
\end{enumerate}
When $QS$ is an ideal of $L^1 \oplus \mathcal T^1$, then we have $I = \overline{I} \cap QS$ in (2).
\end{proposition}
\begin{proof}
The proof is nearly identical to the one given in 
\cite[Thm.~1.1]{Burnham1972}, and is hence omitted.
\end{proof}
We have already seen in Lemma \ref{lem:graded_closed_ideals} that graded closed ideals of $L^1\oplus \mathcal T^1$ admit a particularly simple structure. This is also true for quantum Segal algebras. For completeness, we repeat the definition, which is the same as for closed ideals: A quantum Segal algebra $QS$ is graded if $QS \cong (QS \cap (L^1 \oplus \{ 0\})) \oplus (QS \cap (\{ 0\} \oplus \mathcal T^1))$. It should be noted that not every quantum Segal algebra is graded, see Example \ref{qsa:notgraded}. 
Note that for a graded quantum Segal algebra $(QS, \|\cdot\|_{QS})$, an easy application of the open mapping theorem yields that its norm is equivalent to the sum of the subspace norms:
\begin{align*}
    c_1 (\| (f, 0)\|_{QS} + \| (0, A)\|_{QS}) \leq \| (f, A)\|_{QS} \leq c_2(\| (f, 0)\|_{QS} + \| (0, A)\|_{QS}), \quad (f, A) \in QS.
\end{align*}
Even though we do not yet know if the complete ideal theorem for Segal algebras carries over, on the level of regular maximal ideals this is true for graded quantum Segal algebras.
\begin{proposition}\label{prop: gelfand}
Let $(QS,\|\cdot\|_{QS})$ be a graded quantum Segal algebra.
\begin{enumerate}[(1)]
    \item Every $\chi \in \chi_{\mathbb{R}^{2n}\oplus \mathbb{Z}_2}$ is a multiplicative linear functional on $QS$.
    \item For every $\chi \in \mathcal M(QS)$ there is some $(z, j) \in \mathbb R^{2n} \times \mathbb Z_2$ such that $\chi = \left.\chi_{z, j}\right|_{QS}$.
    \item\label{prop: 3 nonequal} If $(z_1, j_1),\, (z_2, j_2) \in \mathbb R^{2n} \times \mathbb Z_2$ such that $\left.\chi_{z_1, j_1}\right|_{QS} = \left.\chi_{z_2, j_2}\right|_{QS}$ then $(z_1, j_1) = (z_2, j_2)$.
\end{enumerate}
    Hence, the Gelfand transform $\Gamma_{QS}$ of $QS$ is simply the restriction $\Gamma_{L^1\oplus \mathcal T^1}|_{QS}$.
\end{proposition}
\begin{proof}
\begin{enumerate}[(1)]
    \item This is clear, as $QS$ and $L^1 \oplus \mathcal T^1$ have the same algebraic operations.
    \item Notice that $S=QS \cap L^1(\mathbb{R}^{2n})\oplus\{0\}$ is a Segal algebra. Let $\chi\in \mathcal M(QS)$. Then, $\chi|_{S}$ is a nonzero multiplicative linear functional on $S$. As Segal algebras have the same Gelfand theory as $L^1(\mathbb{R}^{2n})$, there is some $z \in \mathbb R^{2n}$ such that
    \begin{align*}
        \chi|_{S}(f) = \widehat{f}(z), \quad f \in S.
    \end{align*}
    From this, one concludes as in the proof of Proposition \ref{prop: mult char} that
    \begin{align*}
        \chi(f, A) = \widehat{f}(z) +(-1)^j \widehat{A}(z)=\chi_{z,j}(f,A).
    \end{align*}
    \item The assumption yields 
    \begin{equation*}
    \widehat{f}(z_1) = \chi_{z_1, j_1}(f, 0) = \chi_{z_2, j_2}(f,0) = \widehat{f}(z_2).
    \end{equation*}
    If $z_1 \neq z_2$, then there exists $f \in L^1(\mathbb R^{2n})$ such that $\widehat{f}(z_1) \neq \widehat{f}(z_2)$ and a sequence $\{f_n\}_{n\in \mathbb{N}} \subset S$ with $f_n \to f$ in $L^1(\mathbb{R}^{2n})$. Since point evaluations of the Fourier transform are continuous on $L^1(\mathbb{R}^{2n})$, we thus would obtain
    \begin{equation*}
        \widehat{f}(z_1) =\lim_{n\to \infty} \widehat{f_n}(z_1) = \lim_{n\to \infty}\widehat{f_n}(z_2) =\widehat{f}(z_2), 
    \end{equation*}
    which is a contradiction. Hence $z_1 = z_2$. Further, if $j_1 \neq j_2$, then $\widehat{A}(z_1) = 0$ for every $A \in QS \cap \{0\}\oplus\mathcal T^1$. Since there exists $B \in \mathcal T^1$ with $\widehat{B}(z_1) \neq 0$, we see that $j_1 = j_2$.\qedhere
\end{enumerate}
\end{proof}

\begin{corollary}
Let $QS$ be a graded quantum Segal algebra. Then, $QS$ is Jacobson semisimple. 
\end{corollary}

\section{Examples of Quantum Segal Algebras}\label{sec: examples of QSA}
While we know that $L^{1} \oplus \mathcal{T}^{1}$ is a quantum Segal algebra, we do not yet have any nontrivial examples. In this section we will look at several examples of quantum Segal algebras.
\subsection{Induced Quantum Segal Algebras}
We say that $A \in \mathcal{T}^{1}$ is a \emph{regular operator} if it  satisfies $\mathcal{F}_{W}(A)(z) \neq 0$ for all $z \in \mathbb{R}^{2n}$. Using this, we have the following construction.
\begin{definition}
Let $(S, \|\cdot\|_{S})$ denote a Segal algebra and fix a regular operator $A \in \mathcal{T}^{1}$. Define the subspace $S^{A} \subset L^{1} \oplus \mathcal{T}^{1}$ as 
\[S^{A} \coloneqq \left\{(f, g \ast A)  :  f, g \in S \right\},\] 
and the norm on $S^{A}$ by 
\begin{equation} \label{eq:norm_on_quantum_feichtinger_algebra}
    \|(f, g \ast A)\|_{S^{A}} \coloneqq \|f\|_{S} + \|A\|_{\mathcal{T}^{1}}\|g\|_{S}.
\end{equation}
We say that $(S^{A}, \|\cdot\|_{S^{A}})$ is the \emph{induced quantum Segal algebra} of $(S, \|\cdot\|_{S})$ and $A$.
\end{definition}
The reason we need to require that $A \in \mathcal{T}^{1}$ is regular will be clear from the proof of the following result.

\begin{theorem}
\label{thm: induced_segal_algebra}
The induced quantum Segal algebra $(S^{A}, \|\cdot\|_{S^{A}})$ of $(S, \|\cdot\|_{S})$ and $A$ is a quantum Segal algebra. Moreover, in the case that $(S, \|\cdot\|_{S})$ is star-symmetric and $A^{\ast_{\textrm{QHA}}} = A$ we have that $(S^{A}, \|\cdot\|_{S^{A}})$ is star-symmetric.
\end{theorem}

The final claim regarding star-symmetry follows immediately from \eqref{eq:involution operator operator} and \eqref{eq:involution function operator}. However, one can give the precise criterion for when $S^A$ is star-symmetric. We quickly discuss this before turning to the proof of Theorem \ref{thm: induced_segal_algebra}:
\begin{proposition}
Let $(S^A, \|\cdot\|_{S^A})$ be an induced quantum Segal algebra. Then it is star-symmetric if and only if the Segal algebra $S$ is star-symmetric and closed under convolution by the tempered distribution $\phi$ defined by
\begin{equation*}
    \varphi = \mathcal F_\sigma( \overline{\mathcal F_W(A)}/\mathcal F_W(A)).
\end{equation*}
\end{proposition}
\begin{proof}
We need to be able to solve the equation
\begin{equation*}
    (f_1, g_1 \ast A)^\ast = (f_2, g_2 \ast A)
\end{equation*}
for $f_2, g_2 \in S$ whenever $f_1, g_1 \in S$ are given. Clearly, $f_2=f_1^\ast$ making $S$ star-symmetric. Applying $\mathcal F_W$ to the second component yields
\begin{equation*}
    \mathcal F_\sigma(g_1^\ast) \mathcal F_W(A^{\ast_{\textrm{QHA}}}) = \mathcal F_\sigma(g_1^\ast) \overline{\mathcal F_W(A)} = \mathcal F_\sigma(g_2) \mathcal F_W(A).
\end{equation*}
Since $\mathcal F_W(A)$ is nowhere zero, we can solve this for
\begin{equation*}
    \mathcal F_\sigma(g_2) = \mathcal F_\sigma(g_1^\ast) \overline{\mathcal F_W(A)}/\mathcal F_W(A),
\end{equation*}
or equivalently
\begin{equation*}
    g_2 = g_1^\ast \ast \mathcal F_\sigma^{-1}(\overline{\mathcal F_W(A)}/\mathcal F_W(A))
\end{equation*}
in the sense of tempered distributions.
\end{proof}
\begin{remark}
When $A = A^{\ast_{\textrm{QHA}}}$ we have $\overline{\mathcal F_W(A)}/\mathcal F_W(A) = 1$ implying that $\varphi = (2\pi)^n\delta_0$. 
\end{remark}

We break the verification of Theorem \ref{thm: induced_segal_algebra} into the following three lemmas:
\begin{lemma}
The space $(S^{A}, \|\cdot\|_{S^{A}})$ is a Banach space.
\end{lemma}

\begin{proof}
The homogeneity and triangle inequality the norm is straightforward. For the positive definiteness, it is clear from \eqref{eq:norm_on_quantum_feichtinger_algebra} that $\|(f, g \ast A)\|_{S^{A}} = 0$ implies that $(f, g \ast A) = (0, 0)$. Conversely, assume that $(f, g \ast A) = (0,0)$. Then $f = 0$ and $g \ast A = 0$. By using \eqref{eq:conv_function_operator} we have that 
\[\mathcal{F}_{W}(g \ast A) = \mathcal{F}_{\sigma}(g)\cdot \mathcal{F}_{W}(A) = 0.\]
Since $A$ is regular, this forces $\mathcal{F}_{\sigma}(g)= 0$. Hence $g = 0$ and $\|\cdot\|_{S^{A}}$ is a norm on $S^{A}$. 
The completeness of $S^A$ follows easily from the completeness of $S$.
\end{proof}

\begin{lemma}
The induced Segal algebra $(S^{A}, \|\cdot\|_{S^{A}})$ satisfies \ref{label:QS1} and \ref{label:QS2}.
\end{lemma}
\begin{proof}
To show \ref{label:QS1}, notice that $L^1(\mathbb R^{2n}) \ast A$ is dense in $\mathcal T^1$ whenever $A$ is regular by \cite{werner84}. By \ref{def:s1} the induced Segal algebra $(S^{A}, \|\cdot\|_{S^{1}})$ is a dense subspace of $L^{1} \oplus \mathcal{T}^{1}$.

For \ref{label:QS2}, let us first show that $\ast$ is a well-defined operator on $S^{A}$. For two elements $(f_{1}, g_{1} \ast A)$ and $(f_{2}, g_{2} \ast A)$ we have 
\[(f_{1}, g_{1} \ast A) \ast (f_{2}, g_{2} \ast A) = (f_{1} \ast f_{2} + (g_{1} \ast g_{2}) \ast (A \ast A), (f_{1} \ast g_{2} + f_{2} \ast g_{1}) \ast A).\]
Since $S$ is closed under convolution, we know that \[f_{1} \ast f_{2},\, g_{1} \ast g_{2},\, f_{1} \ast g_{2} + f_{2} \ast g_{1}\in S.\] Moreover, we know that $A \ast A \in L^{1}(\mathbb{R}^{2n})$. Since Segal algebras are ideals of $L^1(\mathbb{R}^{2n})$ the product is well-defined. 
To see that $(S^{A}, \|\cdot\|_{S^{A}})$ is a Banach algebra we compute
\begin{align*}
    & \|(f_{1}, g_{1} \ast A) \ast (f_{2}, g_{2} \ast A)\|_{S^{A}} \\ 
    & \leq  \|f_{1} \ast f_{2}\|_{S} + \|g_{1} \ast g_{2}\|_{S} \|A \ast A\|_{L^{1}} + \|A\|_{\mathcal{T}^{1}}\|f_{1} \ast g_{2} + f_{2} \ast g_{1}\|_{S} \\ 
    & \leq  \|f_{1}\|_{S} \|f_{2}\|_{S} + \|g_{1}\|_{S} \|g_{2}\|_{S} \|A\|_{\mathcal{T}^{1}}^{2} + \|A\|_{\mathcal{T}^{1}}(\|f_{1}\|_{S} \|g_{2}\|_{S} + \|f_{2}\|_{S} \|g_{1}\|_{S}) \\ 
    & =
    \|(f_{1}, g_{1} \ast A)\|_{S^{A}} \cdot \|(f_{2}, g_{2} \ast A)\|_{S^{A}}.\qedhere
\end{align*}
\end{proof}
\begin{lemma}
The induced Segal algebra $(S^{A}, \|\cdot\|_{S^{A}})$ satisfies \ref{label:QS3} and \ref{label:QS4}.
\end{lemma}

\begin{proof}
For any $z \in \mathbb{R}^{2n}$ we have that
\begin{align*}
\alpha_z(f, g \ast A) = (\alpha_z(f), \alpha_z(g) \ast A).
\end{align*}
Since the Segal algebra $S$ is shift invariant and shifts are norm isometric, it follows that $(S^{A}, \|\cdot\|_{S^{A}})$ is shift invariant and shifts are norm isometric. 
Finally, the shifts act continuously since
\begin{equation*}
\lim_{z\to 0}\| \alpha_z(f, g \ast A) - (f, g \ast A)\|_{S^{A}}
= \lim_{z\to 0}\| \alpha_z(f) - f\|_{S} + \| A\|_{\mathcal{T}^1} \lim_{z\to 0}\| \alpha_z(g) - g\|_{S}= 0. \qedhere
\end{equation*}
\end{proof}

\begin{example}
Let us assume that $A=\phi\otimes \psi$ is a rank one operator and $g\in S\subset L^1(\mathbb{R}^{2n})$, where $S$ is a Segal algebra. Define the \emph{localization operator} by
\[\mathcal{A}^{\phi,\psi}_g\eta \coloneqq \int_{\mathbb{R}^{2n}}g(z)\langle \eta,W_z\psi\rangle W_z\phi\, \mathrm{d}z .\]
The localization operator can be rewritten as
\[g\ast(\phi\otimes \psi)=\mathcal{A}^{\phi,\psi}_g.\]
Hence the induced quantum Segal algebra with respect to $\phi\otimes \psi$ and $S$ is on the form  \[(f,\mathcal{A}^{\phi,\psi}_g), \qquad f,g\in S.\]
To get a quantum Segal algebra we need that $\mathcal{F}_W(\phi\otimes \psi)(z)\not=0$ for all $z\in \mathbb{R}^{2n}$. Computing the Fourier-Weyl transform of $\phi\otimes \psi$ we get
\[\mathcal{F}_W(\phi\otimes\psi)(z) = e^{i x\xi/2}V_{\psi}\phi(-z)\not = 0,\qquad z=(x,\xi)\in \mathbb{R}^{2n},\]
where $V_{\psi}\phi$ is given in \eqref{eq:STFT}. Examples of functions satisfying $V_{\psi}\phi(z)\not = 0$ are given in \cite{GJM20}.
\end{example}

\begin{remark}
Not every quantum Segal algebras is an induced Segal algebra. As an example, let $\mathcal{S}_0=\mathcal{S}_0(\mathbb{R}^{2n})$ denote the Feichtinger algebra defined in Example \ref{ex:segal} \ref{ex:Feichtinger}. Then for any regular operator $A$ the set $L^1(\mathbb{R}^{2n}) \oplus (\mathcal{S}_0 \ast A)$ is a quantum Segal algebra with the norm
\[\|(f,g\ast A)\|_{L^1(\mathbb{R}^{2n}) \oplus (\mathcal{S}_0 \ast A)}=\|f\|_{L^1}+\|g\|_{\mathcal{S}_0}\|A\|_{\mathcal{T}^1}.\]
\end{remark}
\begin{remark}\label{remex}
Induced Segal algebras are in general not ideals of $L^1 \oplus \mathcal T^1$. As an example, let $n=1$ and consider the induced Segal algebra $L^1(\mathbb R^2)^A$ with $A \in \mathcal T^1$ a regular operator. Then the set
\begin{align*}
    X=\{ f \ast A: f \in L^1(\mathbb R^2)\} = \{ B \in \mathcal T^{1}: (0, B) \in L^1(\mathbb R^2)^A\}
\end{align*}
is dense in $\mathcal T^1$. Further, $X$ is a proper subset of $\mathcal T^1$, since $A\not\in X$. This follows from $A = 2\pi\delta_0 \ast A$ together with the map $\mu \mapsto \mu \ast A$ being injective for all $\mu\in \mathcal{S}'(\mathbb{R}^{2})$. Nevertheless, since $\mathcal T^1$ is an essential $L^1$ module, we have
\begin{align*}
f\ast B=A, \quad f\in L^1(\mathbb{R}^2),\, B\in \mathcal{T}^1.
\end{align*}
We therefore conclude that $L^1(\mathbb R^2)^A$ is not an $L^1 \oplus \mathcal T^1$ ideal, since $(f, 0)\in L^1(\mathbb R^2)^A$ and \[(0,B)\ast (f,0)=(0,A)\not \in L^1(\mathbb{R}^2)^A.\]
\end{remark}

\begin{remark}
One might initially expect that if $S$ is a modulation-invariant Segal algebra, then the induced quantum Segal algebra $S^{A}$ is modulation-invariant as well. However, this is not the case. As a simple counterexample, let $S = \mathcal{S}_{0}(\mathbb{R}^{2})=\mathcal{S}_{0}$ be the Feichtinger algebra and let $A$ be the regular operator given by 
\[\mathcal{F}_{W}(A)(z) = e^{-z^2}, \qquad z \in \mathbb{R}^{2}.\] 
Assume by contradiction that $S^{A}$ is modulation-invariant. Then for any $g_{1} \in \mathcal{S}_{0}$ and any $z' \in \mathbb{R}^{2}$ there should exist $g_{2} \in \mathcal{S}_{0}$ such that
\[\gamma_{z'}(g_{1} \ast A) = \gamma_{z'}(g_{1}) \ast \gamma_{z'}(A) = g_{2} \ast A.\]
Taking the Fourier-Weyl transform of both sides gives
\[\mathcal{F}_{\sigma}(g_{1})(z - z')e^{-(z - z')^2} = \mathcal{F}_{\sigma}(g_{2})(z)e^{-z^2}.\]
Since $\mathcal{F}_{\sigma}(\mathcal{S}_0)=\mathcal{S}_0$, we have that
$h_{1} = e^{-z'^2}\mathcal{F}_{\sigma}(g_{1}) \in \mathcal{S}_{0}$ and $h_{2} = \mathcal{F}_{\sigma}(g_{2}) \in \mathcal{S}_{0}$. By picking $z' = -\frac{1}{2}$ we obtain
\[h_{1}\left(z + \frac{1}{2}\right)e^{z} = h_{2}(z).\]
Since $h_{1}$ is arbitrary, and $e^z\mathcal{S}_{0}\not = \mathcal{S}_{0}$ we have a contradiction. For example, for $h_{1}(z) = (1 + |z|^2)^{-1}$ we have no solution $h_2\in \mathcal{S}_0$.
\end{remark}

\subsection{Quantum Segal Algebras through Quantization}
In this section, we will construct quantum Segal algebras using the Weyl quantization. Recall that $\widetilde{f}(z)=f(-z)$, which naturally extends to $\mathcal S(\mathbb R^{2n})$. Given a set  $M$ of functions or distributions, we will use the notation \[\widetilde{M} = \{\widetilde{f} : f\in M\}.\] Notice that if $S$ is a Segal algebra, then so is also $\widetilde{S}$.
Given a set of tempered operators $\mathcal{A} \subset \mathcal S'(\mathcal H)$, we denote the set of symbols by $\mathrm{Sym}(\mathcal{A}) \subset \mathcal S'(\mathbb R^{2n})$, i.e., 
\[\widetilde{\mathrm{Sym}(\mathcal A)} = \mathcal F_\sigma(\mathcal F_W(\mathcal A )).\] Vice versa, given a set of tempered distributions $S \subset \mathcal S'(\mathbb R^{2n})$, we denote by $\mathrm{Sym}^{-1}(S)$ the set of quantized operators as a subset of $\mathcal S'(\mathcal H)$, i.e.,
\[\mathrm{Sym}^{-1}(S) = P\mathcal F_W( \mathcal F_\sigma(\widetilde{S}))P.\]
In the following result, we present a Segal algebra which plays an important role for the discussion of quantum Segal algebras.
\begin{lemma}
    Denote
    \begin{align*}
        \mathcal T := \{ f \in L^1(\mathbb R^{2n}): ~A_f \in \mathcal T^1\}.
    \end{align*}
    Endowed with the norm
    \begin{align*}
        \| f\|_{\mathcal T} := \| f\|_{L^1} + \| A_f\|_{\mathcal T^1},
    \end{align*}
   the space $\mathcal T$ is a star-symmetric and strongly modulation-invariant Segal algebra.
\end{lemma}
\begin{proof}
    Firstly, note that $\mathcal T$ contains $\mathcal S(\mathbb R^{2n})$, so it is dense in $L^1(\mathbb R^{2n})$. We have
    \begin{align*}
        \| \alpha_z(f)\|_{\mathcal T} = \| \alpha_z(f)\|_{L^1} + \| A_{\alpha_z(f)}\|_{\mathcal T^1} = \| \alpha_z(f)\|_{L^1} + \| \alpha_{-z}(A_f)\|_{\mathcal T^1} = \| f\|_{\mathcal T}.
    \end{align*}
    Similarly, the shifts act strongly continuous on $\mathcal T$. Further, it is
    \begin{align*}
        \| f\ast g\|_{\mathcal T} &= \| f \ast g\|_{L^1} + \| A_{f \ast g}\|_{\mathcal T^1} = \| f\ast g\|_{L^1} + \| \widetilde{f} \ast A_g\|_{\mathcal T^1} \\
        &\leq \| f\|_{L^1} \| g\|_{L^1} + \| f\|_{L^1} \| A_g\|_{\mathcal T^1} \leq \| f\|_{\mathcal T} \| g\|_{\mathcal T}
    \end{align*}
    for $f, g \in \mathcal T$. Therefore, $\mathcal T$ is a Segal algebra. Regarding star-symmetry, note that 
    \begin{align*}
        A_{f^\ast} = A_{\tilde{\overline{f}}} = P(A_{\overline{f}})P = PA_f^\ast P = A_f^{\ast_{\textrm{QHA}}}.
    \end{align*}
    Since $\| A_f\|_{\mathcal T^1} = \| A_f^{\ast_{\textrm{QHA}}}\|_{\mathcal T^1}$, and clearly $\| f\|_{L^1} = \| f^\ast\|_{L^1}$, we have $\| f\|_{\mathcal T} = \| f^\ast\|_{\mathcal T}$. The strong modulation invariance is now immediate from Lemma \ref{lemma:modulation_cont_t1}.
\end{proof}

\begin{proposition}\label{prop:segal_parts}
Let $QS =S_{1} \oplus S_{2} \subset L^{1}(\mathbb{R}^{2n}) \oplus \mathcal{T}^{1}$ be a graded quantum Segal algebra. Then both $S_{1}$ and $\mathrm{Sym}(S_{2}) \cap L^{1}(\mathbb{R}^{2n})$ are Segal algebras. Additionally, \[\widetilde{\mathrm{Sym}(S_{2})} \ast \widetilde{\mathrm{Sym}(S_{2})} \subset S_{1}.\]
In particular, when $\mathrm{Sym}(S_{2}) \subset L^{1}(\mathbb{R}^{2n})$ the set $\mathrm{Sym}(S_{2})$ is a Segal algebra.
\end{proposition}
\begin{proof}
It is clear that $S_1$ is a Segal algebra.
$S = \mathrm{Sym}(S_{2}) \cap L^{1}(\mathbb{R}^{2n})$ is a Banach space with the shift-invariant norm 
\[\|f\|_{S}=\|(0, A_f)\|_{QS}+\|f\|_{L^1},\qquad f\in S.\] 
The set $S$ is a Banach algebra, since by Proposition \ref{prop:L^1 module} and \eqref{eq: quantization of convolution} we have \[\widetilde{S}\ast\widetilde{S}\subset \widetilde{S}. \]
It remains to show that $S$ is dense in $L^1(\mathbb R^{2n})$. For this, note that
\begin{align*}
    \mathcal T \ast S_2 \subset L^1(\mathbb R^{2n}) \ast S_2 = S_2.
\end{align*}
Further, we have
\begin{align*}
    \mathcal{T}\ast S_2=\mathrm{Sym}^{-1}(\mathrm{Sym}^{-1}(\mathcal{T}) \ast \tilde{S_2})\subset \mathrm{Sym}^{-1}( \mathcal{T}^1\ast \tilde{S_2})\subset\mathrm{Sym}^{-1}( \mathcal{T}^1\ast \mathcal{T}^1)\subset \mathrm{Sym}^{-1}(L^1(\mathbb{R}^{2n})),
\end{align*}
where $\tilde{S_2}=\mathrm{Sym}^{-1}(\widetilde{\mathrm{Sym}(S_2)})$. Together, this shows $\mathcal T \ast S_2 \subset S$.
Note that, since $S_2$ is dense in $\mathcal T^1$, either by Proposition \ref{prop: gelfand} \ref{prop: 3 nonequal} or by Wiener's approximation theorem for $\mathcal T^1$, cf.~\cite{werner84}, we have that \[ \{z\in \mathbb{R}^{2n}:\mathcal{F}_W(A)(z)=0 \text{ for every } A \in S_2 \}=\emptyset .\]
Now using \eqref{eq:conv_function_operator} together with the fact that for every $z\in \mathbb R^{2n}$ there exist $f\in \mathcal{T}$ and $A\in S_2$ such that $\mathcal{F}_\sigma(f)\not =0$ and $ \mathcal{F}_W(A)\not = 0$, we get that
 \[\{z\in \mathbb{R}^{2n}:\mathcal{F}_W(A)(z)=0 \text{ for every } A \in S \}=\emptyset.\]
The set $\mathcal{F}_W(S)$ is an ideal of $\mathcal{F}_W(\mathrm{Sym}^{-1}(\mathcal{T}))\subset C_0(\mathbb{R}^{2n})$ with pointwise multiplication.
Since $\mathcal{T}$ is a Segal algebra, the set $\mathcal{F}_\sigma( \mathcal{T})$ is a standard algebra by \cite[Re.\ 2.1.15]{reiter20}.
Hence by \cite[Prop.\ 2.1.14]{reiter20} we have that \[C_c(\mathbb{R}^{2n})\cap \mathcal{F}_\sigma(L^1(\mathbb{R}^{2n}))\subset \mathcal{F}_W(S_2).\]
Since $\mathcal{F}_\sigma(C_c(\mathbb{R}^{2n})\cap \mathcal{F}_\sigma(L^1(\mathbb{R}^{2n})))$ is dense in $L^1(\mathbb{R}^{2n})$ the space $S$ is a Segal algebra.
The product structure on $QS$ together with \eqref{eq: quantization of convolution} gives
\[\widetilde{\mathrm{Sym}(S_{2})}\ast\widetilde{\mathrm{Sym}(S_{2})} \subset S_{1}.\qedhere\]
\end{proof}

The following result shows how to generate a wealth of examples of graded quantum Segal algebras.
\begin{proposition}\label{Prop:construction}
Let $S_{1}$ and $S_{2}$ be two Segal algebras where $\widetilde{S_{2}}\subset S_{1}$. Then \[QS = S_{1} \oplus (\mathrm{Sym}^{-1}(S_{2}) \cap \mathcal T^1)\] is a quantum Segal algebra with the norm
\[\|(f,A_g)\|_{QS}= \|f\|_{S_{1}}+C\|\widetilde{g}\|_{S_{2}} + C'\| A_g\|_{\mathcal T^1},\]
for positive constants $C$, $C'$.
\end{proposition}
\begin{proof}
For the property \ref{label:QS1}, note that the set $A_0 := L^1(\mathbb R^{2n}) \cap \mathcal F_\sigma(C_c(\mathbb R^{2n}))$ is contained in $\mathcal T$, since $\mathcal T$ is a Segal algebra, cf.~\cite[Prop.~6.2.5]{reiter20}. Further, the set is dense in $\mathcal T$. Since $\mathrm{Sym}^{-1}(\mathcal T)$ is dense in $\mathcal T^1$, we obtain that $\mathrm{Sym}^{-1}(A_0)$ is dense in $\mathcal T^1$. But $A_0$ is also contained in $S_2$. Therefore, $\mathrm{Sym}^{-1}(S_2) \cap \mathcal T^1$ is dense in $\mathcal T^1$.

The conditions \ref{label:QS3} and \ref{label:QS4} are straightforward to verify. For \ref{label:QS2}, $QS$ is easily seen to be a complete subspace of $L^1 \oplus \mathcal T^1$. The product is well-defined by \eqref{eq: quantization of convolution}. Hence, we only need to show that the Banach algebra property holds, that is for $(f_{1}, A_{g_{1}}), (f_{2}, A_{g_{2}}) \in QS$ we have
\[\|(f_1,A_{g_1}) \ast (f_2,A_{g_2})\|_{QS} \leq \|(f_1,A_{g_1})\|_{QS} \cdot \|(f_2,A_{g_2})\|_{QS}.\]
We know by \eqref{eq:banach_ideal_bound_general} and Lemma \ref{lem:homeo_implies_cont} that there exist positive constants $c_1$ and $c_2$ such that 
\begin{align*}
    \|g\|_{S_{1}} & \le c_1 \cdot \|\widetilde{g}\|_{S_{2}} \\
    \|f\ast\widetilde{g}\|_{S_{2}} & \le c_2 \cdot \|f\|_{S_{1}}\|\widetilde{g}\|_{S_{2}}
\end{align*}
for all $f\in S_{1}$ and $g\in S_{2}$. 
Denote by $C=\max \{c_1, c_2\}$. Further, by Lemma \ref{lem:homeo_implies_cont}, there is $C' > 0$ with $\| f\|_{L^1} \leq C' \| f\|_{S_1}$ for $f \in S_1$. We get 
\begin{align*}
   \|(f_1,A_{g_1}) \ast (f_2,A_{g_2})\|_{QS}
   &= \|f_1\ast f_2+\widetilde{g_1}\ast\widetilde{g_2}\|_{S_{1}}+\|f_1\ast\widetilde{g_2}+f_2\ast\widetilde{g_1}\|_{S_{2}} \\
   &\quad \quad + \| f_1 \ast A_{g_2} + f_2 \ast A_{g_1}\|_{\mathcal T^1}\\
   &\le \|f_1\|_{S_{1}}\|f_2\|_{S_{1}}+C^2\|\widetilde{g_1}\|_{S_{2}}\|\widetilde{g_2}\|_{S_{2}}+C\|f_1\|_{S_{1}}\|\widetilde{g_2}\|_{S_{2}} \\
   &\quad \quad + C\|f_2\|_{S_{1}}\|\widetilde{g_1}\|_{S_{2}} + \| f_1\|_{L^1} \| A_{g_2}\|_{\mathcal T^1} + \| f_2\|_{L^1} \| A_{g_1}\|_{\mathcal T^1}\\
   &\leq\|(f_1,A_{g_1})\|_{QS} \cdot \|(f_2,A_{g_2})\|_{QS}.\qedhere
\end{align*}
\end{proof}
\begin{remark}\hfill
\begin{enumerate}
    \item If, under the assumptions of the previous proposition, we further have $\mathrm{Sym}^{-1}(S_2)\subset \mathcal T^1$, then one can instead let $C' = 0$, i.e.,~the norm becomes $\| (f, A_g)\|_{QS} = \| f\|_{S_1} + C\| \widetilde{g}\|_{S_2}$. This follows from essentially the same proof, simply omitting the additional terms.
\item By the previous result, we can consider
\begin{align*}
    \mathcal T^Q := \mathcal T \oplus \mathrm{Sym}^{-1}(\mathcal T),
\end{align*}
which is a star-symmetric, strongly modulation-invariant quantum Segal algebra upon being endowed with the norm
\begin{align*}
    \| (f, A_g)\|_{\mathcal T^Q} = \| f\|_{\mathcal T} + \| g\|_{\mathcal T}.
\end{align*}
While we won't use this particular quantum Segal algebra in the following, it seems that it is a convenient framework to work within.
\item When we let $S_2 = \widetilde{S_1}$ in the above proposition, then simple computations show that the resulting quantum Segal algebra $QS = S_1 \oplus \mathrm{Sym}^{-1}(\widetilde{S_1})$ is a module over $\mathcal T^Q$.
\end{enumerate}
\end{remark}

\begin{corollary}\label{intersection_all_qsa}
Let $G$ denote the set of all graded quantum Segal algebras. Then 
\[\mathcal{F}\Big(\bigcap_{S_1\oplus S_2\in G} S_1\oplus S_2\Big) = (C_{c}(\mathbb{R}^{2n})\cap \mathcal{F}_\sigma(L^1(\mathbb{R}^{2n}))) \oplus (C_{c}(\mathbb{R}^{2n})\cap \mathcal{F}_\sigma(L^1(\mathbb{R}^{2n}))) .\]
\end{corollary}
\begin{proof}
We know that the intersection of all Segal algebras satisfies
\[\mathcal{F}_{\sigma}\Big(\bigcap_{S\in \mathcal{S}} S\Big) = C_{c}(\mathbb{R}^{2n}) \cap \mathcal{F}_\sigma(L^1(\mathbb{R}^{2n})),\]
where $\mathcal{S}$ denotes the set of all Segal algebras.
By the proof of Proposition \ref{prop:segal_parts} we know that 
\[\mathcal{F}\Big(\bigcap_{S_1\oplus S_2\in G} S_1\oplus S_2\Big) \supseteq (C_{c}(\mathbb{R}^{2n})\cap \mathcal{F}_\sigma(L^1(\mathbb{R}^{2n}))) \oplus (C_{c}(\mathbb{R}^{2n})\cap \mathcal{F}_\sigma(L^1(\mathbb{R}^{2n}))).\]

The remaining inclusion follows from Proposition \ref{Prop:construction}. Note that $\mathrm{Sym}^{-1}( S_2) \subset \mathcal T^1$ can always be enforced by intersecting $S_2$ with the Feichtinger  algebra $\mathcal S_0$  which satisfies $\mathrm{Sym}^{-1}(\mathcal{S}_0) \subset \mathcal T^1$, see \cite[Thm.~3.5]{heil2008}.
\end{proof}

\subsection{The Quantum Feichtinger Algebra}
In this section we will describe a quantum Segal algebra that we call the \emph{quantum Feichtinger algebra}. It is a particular example of the quantum Segal algebras obtained by quantization of the Feichtinger algebra. We will show that this algebra is not on the form of an induced Segal algebra. Recall that $\mathcal{S}_{0} \coloneqq \mathcal{S}_{0}(\mathbb{R}^{2n})$ denotes the Feichtinger algebra defined in Example~\ref{ex:Feichtinger}. 

\begin{definition}
We say that a bounded linear operator $A$ on $\mathcal H$ is a \emph{Feichtinger operator} if $A = A_{g}$ for $g \in \mathcal{S}_{0}(\mathbb{R}^{2n})$. We denote the \emph{space of Feichtinger operators} as $\mathcal{S}_0(\mathcal H)$.
\end{definition}

Since $\mathcal S_0(\mathbb R^{2n}) \subset L^2(\mathbb R^{2n})$, there is no complication in talking about the integral kernel of a Feichtinger operator. As the discussions in \cite[Sec.~7.4]{Heil2003} shows, the integral kernel is contained in $\mathcal S_0(\mathbb R^{2n})$ if and only if the Weyl symbol is in $\mathcal S_0(\mathbb R^{2n})$, and, in this case, their norms in $\mathcal S_0(\mathbb R^{2n})$ are equivalent. Therefore, $A_{g} \in \mathcal{T}^{1}$ whenever $g \in \mathcal{S}_{0}(\mathbb{R}^{2n})$, which follows from \cite[Thm.~3.5]{heil2008}. Clearly, the space $\mathcal S_0(\mathcal H)$ is a Banach space under the norm 
\[\|A_{g}\|_{\mathcal{S}_0(\mathcal{H})} \coloneqq \|g\|_{\mathcal{S}_{0}}, \qquad A_{g} \in \mathcal{S}_0(\mathcal{H}).\]
We will let $\textrm{Fin}(\mathcal{S}_{0})$ denote all finite-rank operators $F$ satisfying $\|F\|_{\textrm{Fin}(\mathcal{S}_{0})} < \infty$, where
    \[\|F\|_{\textrm{Fin}(\mathcal{S}_{0})} \coloneqq \inf \left\{\sum_{i=1}^m|\alpha_{i}|\|\phi_{i}\|_{\mathcal{S}_{0}}\|\psi_{i}\|_{\mathcal{S}_{0}} : F = \sum_{i=1}^m\alpha_{i} \cdot \phi_{i} \otimes \psi_{i}\right\}.\]
The Feichtinger operators inherits several equivalent characterizations from the Feichtinger algebra.
\begin{theorem}\label{thm:equivalent_feichtinger_characterizations}
Let $A \in \mathcal{T}^{1}$. Then the following characterizations are equivalent:
\begin{enumerate}[1)]
    \item\label{thm:equiv_feichtinger 1} The operator $A$ is a Feichtinger operator.
    \item\label{thm:equiv_feichtinger 2} The integral kernel $K_{A}$ of $A$ is in the Feichtinger algebra.
    \item\label{thm:equiv_feichtinger 3} The operator $A$ is in the closure of $\textrm{Fin}(\mathcal{S}_{0})$ with the norm $\|\cdot\|_{\textrm{Fin}(\mathcal{S}_{0})}$.
    \item\label{thm:equiv_feichtinger 4} The operator $A$ satisfies $\mathcal{F}_{W}(A) \in \mathcal{S}_{0}$.
    \item \label{thm:equiv_feichtinger 5}The modulation of $A$ satisfies $\int_{\mathbb R^{2n}} \| (\gamma_z A) \ast A\|_{L^1} \, \mathrm{d}z < \infty$.
    \item \label{thm:equiv_feichtinger 6} As a function of $(z, z') \in \mathbb R^{2n}\times \mathbb R^{2n}$, it is $\tr(A (\alpha_z \gamma_{z'}A)^*) \in L^1(\mathbb R^{2n} \times \mathbb R^{2n})$.
\end{enumerate}
    Introduce the notation 
    \[ \| A\|_{B, \gamma}\coloneqq \int_{\mathbb R^{2n}} \| (\gamma_z A) \ast B\|_{L^1} \, \mathrm{d}z, \quad \| A\|_{B, \alpha\gamma} \coloneqq \| \tr(A (\gamma_{z'}\alpha_z B)^*)\|_{L^1(\mathbb R^{2n} \times \mathbb R^{2n})}, \]
    where $B \in \mathcal{S}_0(\mathcal H)\setminus \{0\}$ is a fixed operator.
    Then $\mathcal{S}_0(\mathcal H)$ can be given the following equivalent norms
    \[\|A\|_{\mathcal{S}_0(\mathcal{H})} \cong \| \mathcal F_W(A)\|_{\mathcal S_0} \cong \| K_A\|_{\mathcal S_0} \cong \| A\|_{Fin(\mathcal S_0)} \cong \| A\|_{B, \gamma} \cong \| A\|_{B, \alpha\gamma}. \]
\end{theorem}
\begin{remark}
Note that $V_B(A)(z, z') \coloneqq \tr(A\alpha_z \gamma_{z'}(B))$ serves as an operator analogue to the short-time Fourier transform. To the authors' knowledge, the operator STFT is not present in the literature. It could serve as an interesting object for further studies.
\end{remark}
\begin{proof}
The equivalence between \ref{thm:equiv_feichtinger 1} and \ref{thm:equiv_feichtinger 2} has already been discussed above. The equivalence between \ref{thm:equiv_feichtinger 2} and \ref{thm:equiv_feichtinger 3} is a reformulation of the tensor factorization property of the Feichtinger algebra, see e.g., \cite[Thm.~7.4]{jakobsen18}. Since for $A=A_g$ we have $\mathcal{F}_{\sigma}(\widetilde{g}) = \mathcal{F}_{W}(A_g)$ and $\mathcal{F}_{\sigma}$ is a bijective isometry on $\mathcal{S}_{0}$ the equivalence between \ref{thm:equiv_feichtinger 1} and \ref{thm:equiv_feichtinger 4} follows. For the equivalence between \ref{thm:equiv_feichtinger 1} and \ref{thm:equiv_feichtinger 5} we have
\[(\gamma_z A_{g})\ast A_{g} =  A_{\gamma_{-z}g}\ast A_{g} = \widetilde{(\gamma_{-z}g) \ast g}.\]
Hence \ref{thm:equiv_feichtinger 5} holds if and only if 
\[\int_{\mathbb R^{2n}} \| \gamma_z g \ast g\|_{L^1} \, \mathrm{d}z < \infty.\]
This is a reformulation of $g$ being in $\mathcal{S}_{0}$, see \cite[Thm.~4.7]{jakobsen18}. 
Finally, for seeing that \ref{thm:equiv_feichtinger 4} and \ref{thm:equiv_feichtinger 6} are equivalent, write $f_1 = \mathcal F_W(A)$, and $f_2 = \mathcal F_W(B)$. Then
\begin{align*}
    \tr(A (\alpha_z \gamma_{z'}B)^*) &= \langle A,  \alpha_z \gamma_{z'}B\rangle_{\mathcal{T}^2} = \langle \mathcal{F}_W( A), \mathcal F_W(\alpha_z \gamma_{z'}B)\rangle_{L^2} = \langle f,   \gamma_{z}\alpha_{z'} g\rangle_{L^2}\\
    &= \int_{\mathbb R^{2n}} f(v) e^{-i\sigma(z, v)}\overline{g}(v-z')\,\mathrm{d}v.
\end{align*}
It is not hard to see that
\begin{align*}
    \int_{\mathbb R^{2n}} \int_{\mathbb R^{2n}} \left| \int_{\mathbb R^{2n}} f(v) e^{-i\sigma(z, v)}\overline{f}(v-z')\,\mathrm{d}v \right| \,\mathrm{d}z\,\mathrm{d}w = \int_{\mathbb R^{2n}} \int_{\mathbb R^{2n}} | V_f(f)(z',z) | \,\mathrm{d}z\,\mathrm{d}w,
\end{align*}
where $V_f(f)$ denotes the STFT. 

The equivalences of the norms are straightforward and from the analogous results for the Feichtinger algebra.
\end{proof}

We can now combine the spaces $\mathcal{S}_{0}$ and $\mathcal{S}_0(\mathcal H)$ to form the following algebra.

\begin{definition}
The \emph{quantum Feichtinger algebra} $\mathcal{S}_{0}^{Q} \subset L^{1} \oplus \mathcal{T}^{1}$ is given by
\begin{equation*}
    \mathcal{S}_{0}^{Q}  \coloneqq \mathcal{S}_{0} \oplus \mathcal{S}_0(\mathcal H) = \left\{(f, A_g) \in L^{1} \oplus \mathcal{T}^{1} : f,g \in \mathcal{S}_{0}\right\} 
\end{equation*}
with the norm
\[\|(f,A_g)\|_{\mathcal{S}_0^Q}=\|f\|_{\mathcal{S}_0}+\|g\|_{\mathcal{S}_0}.\]
\end{definition}
We say that a Segal algebra $(QS, \|\cdot \|_{QS})$ is \emph{strongly modulation invariant} if $\| \gamma_z(f, A)\|_{QS} = \| (f, A)\|_{QS}$ and $z \mapsto \gamma_z$ are continuous in the strong sense on $(QS, \|\cdot \|_{QS})$.
\begin{proposition}
The quantum Feichtinger algebra $\mathcal{S}_{0}^{Q}$ is a star-symmetric and strongly modulation-invariant quantum Segal algebra.
\end{proposition}
\begin{proof}
The fact that $\mathcal{S}^Q_0$ is a quantum Segal algebra follows from Proposition \ref{Prop:construction}. 
The star-symmetry follows from the property
\[(f, A_{g})^{\ast} = (f^{\ast}, PA_{g}^{\ast}P) = (f^{\ast}, A_{g^*}).\]
Finally, the strong modulation invariance is immediate since for $z \in \mathbb{R}^{2n}$ and $f, g \in \mathcal{S}_{0}$ we have \[\gamma_{z}(f, A_{g}) = (\gamma_{z}(f), \gamma_{z}(A_{g})) = (\gamma_{z}(f), A_{\gamma_{-z}g}). \qedhere\]
\end{proof}

The Feichtinger algebra is yet another example of a quantum Segal algebra which is not of the induced type:

\begin{proposition}
The quantum Feichtinger algebra is not an induced quantum Segal algebra.
\end{proposition}
\begin{proof}
If $\mathcal{S}_0^Q$ was an induced quantum Segal algebra there would exist a regular trace class operator $A$ such that every Feichtinger operator $B$ can be written as 
\[B=f*A,\]
for some $f\in \mathcal{S}_0$. By Theorem~\ref{thm:equivalent_feichtinger_characterizations} \ref{thm:equiv_feichtinger 4} together with the fact that $\mathcal{F}_\sigma(\mathcal{S}_0)=\mathcal{S}_0$ this is equivalent to that every function $g\in \mathcal{S}_0$ can be written uniquely as 
\[g=f\cdot\mathcal{F}_W(A),\]
for some $f\in \mathcal{S}_0$.
Or equivalently, the map $\phi(f)= f/\mathcal{F}_W(A)$ is a bijection from $\mathcal{S}_0$ to itself. We will show that the surjectivity of $\phi$ contradicts the fact that $\mathcal{F}_W(A)\in L^2(\mathbb{R}^{2n})$.

Denote by 
\[h(z)=\prod_{j=1}^n\frac{1}{1+|x_j|^2}\frac{1}{1+|\xi_j|^2}\in \mathcal{S}_0, \qquad z=(x,\xi)\in \mathbb{R}^{2n}.\]
Then consider $\phi^4(h)=h/(\mathcal{F}_W(A))^4\in \mathcal{S}_0$. 
Since $\mathcal{S}_0\subset C_0(\mathbb{R}^{2n})$ there exists a $C>0$ such that for all $z\in \mathbb{R}^{2n}$ we have 
\[|h(z)|/|\mathcal{F}_W(A)(z)|^{4}\le C,\]
or equivalently
\[\sqrt{|h(z)|}\le C^{1/4}|\mathcal{F}_W(A)(z)|^{2}.\]
However, $\sqrt{|h(z)|}\not \in L^1(\mathbb{R}^{2n})$ implying that $\mathcal{F}_W(A)\not \in L^2(\mathbb{R}^{2n})$.
\end{proof}
\begin{remark}
The Feichtinger algebra $\mathcal S_0(\mathbb R^{n})$ is well-known for being the smallest strongly modulation-invariant Segal algebra. The same is indeed true for graded quantum Segal algebras: Denote by $(X,\|\cdot \|_X)$ a Banach space where $X \subset L^1 \oplus \mathcal T^1$ is a graded subspace, which is both strongly shift- and strongly modulation-invariant. Further, assume that there is at least one element $(f, A) \in X$ with both $f \neq 0$ and $A \neq 0$ and $(f, A) \in \mathcal{S}_0^Q$. By Corollary~\ref{intersection_all_qsa} such elements always exists for $X$ a graded quantum Segal algebra. Then, we endow $\mathcal F(X) \subset C_0(\mathbb R^{2n}) \oplus C_0(\mathbb R^{2n})$ with the norm $\| (\mathcal F_\sigma(f), \mathcal F_W(A))\| = \| (f, A)\|_X$. This turns both components of $\mathcal F(X) $ into a strongly shift- and strongly modulation-invariant space that contains at least one nontrivial element of $\mathcal S_0(\mathbb R^{2n})$. By \cite[Thm.~7.3]{jakobsen18} we thus have $\mathcal S_0(\mathbb R^{2n}) \oplus \mathcal S_0(\mathbb R^{2n}) \subset \mathcal F(X)$.
\end{remark}

One of the useful applications of the Feichtinger algebra is as a space of test functions. Its dual space $\mathcal{S}_0'(\mathbb R^{2n})$ is sometimes referred to as the space of \emph{mild distributions}. It contains, in addition to all $L^p$ spaces, some distributions such as the point measures $\delta_x$ and Dirac combs $\sum_{k \in \mathbb Z^{d}} \delta_k$. Having defined the Feichtinger operators, we can also consider the Banach space dual of $\mathcal S_0(\mathcal H)$, which we will denote by $\mathcal S_0'(\mathcal H)$. Since the Schwartz operators $\mathcal S(\mathcal H)$ embeds continuously into $\mathcal S_0(\mathcal H)$, the dual $\mathcal S_0'(\mathcal H)$ is continuously embedded into $S'(\mathcal H)$. As is the case for the Feichtinger operators, the space $\mathcal S_0'(\mathcal H)$ can be equipped with many equivalent norms. As an example, define the norm
\begin{align*}
    \| A_g\|_{\mathcal S_0'} \coloneqq \| g \|_{\mathcal S_0'(\mathbb R^{2n})}. 
\end{align*}
Here, $\| \cdot\|_{\mathcal S_0'(\mathbb R^{2n})}$ is any of the equivalent norms of $\mathcal S_0'(\mathbb R^{2n})$.
We will denote by $\langle \cdot, \cdot\rangle_{\mathcal{S}_0',\mathcal{S}_0}$ the sesquilinear pairing of an element in $\mathcal{S}_0'(\mathcal{H}) $ and  $\mathcal{S}_0(\mathcal{H}) $.
The following statement is, in principle, the dual statement of Theorem \ref{thm:equivalent_feichtinger_characterizations}.
\begin{theorem}
Let $A \in \mathcal S'(\mathcal H)$ and let $B \in \mathcal{S}_0(\mathcal H)\setminus\{0\}$ be any fixed Feichtinger operator. Denote by 
    \[ \| A\|_{\infty, B, \gamma}\coloneqq \sup_{z \in \mathbb R^{2n}} \| (\gamma_z A) \ast B\|_{L^1}, \quad \| A\|_{\infty, B, \gamma\alpha} \coloneqq \sup_{z,z' \in \mathbb R^{2n}} |\langle A,  \gamma_{z'}\alpha_z(B)\rangle_{\mathcal{S}_0',\mathcal{S}_0}|. \]
Then, the following characterizations are equivalent:
\begin{enumerate}[1)]
    \item\label{thm:equiv_mild 1} The operator $A$ can be written as $A = A_{g}$ for some $g \in \mathcal{S}_{0}'(\mathbb{R}^{2n})$.
    \item\label{thm:equiv_mild 2} There exists an element $K_{A}\in \mathcal{S}_{0}'(\mathbb{R}^{2n})$ such that $\langle Af_2, f_1\rangle_{\mathcal{S}_0',\mathcal{S}_0}= \langle K_A, f_1\otimes f_2\rangle_{\mathcal{S}_0',\mathcal{S}_0}$ for all $f_1, f_2\in \mathcal{S}_0(\mathbb{R}^n)$.
    \item\label{thm:equiv_mild 3} The Fourier-Weyl transform of $A$ satisfies $\mathcal{F}_{W}(A) \in \mathcal{S}_{0}'(\mathbb{R}^{2n})$.
    \item\label{thm:equiv_mild 4} $A$ satisfies $\| A\|_{\infty, B, \gamma} < \infty$.
    \item\label{thm:equiv_mild 5} $A$ satisfies $\sup_{z,z' \in \mathbb R^{2n}} \| A\|_{\infty, B, \gamma\alpha} < \infty$.
\end{enumerate}
    Moreover, the following are equivalent norms on $\mathcal{S}_0'(\mathcal H)$:
    \[\|A\|_{\mathcal{S}_0'} \cong \| \mathcal F_W(A)\|_{\mathcal S_0'} \cong \| K_A\|_{\mathcal S_0'} \cong \| A\|_{\infty, B, \gamma} \cong \| A\|_{\infty, B, \gamma\alpha}. \]
\end{theorem}
We omit the proof, as it is very similar to that of Theorem \ref{thm:equivalent_feichtinger_characterizations}, replacing characterizations of $\mathcal S_0(\mathbb R^{2n})$ with those of $\mathcal S_0'(\mathbb R^{2n})$, cf.\ \cite{jakobsen18}. 
Note that $\mathcal L(\mathcal H)$ continuously embeds into $\mathcal S_0'(\mathcal H)$. Further, by the kernel theorem for $\mathcal S_0(\mathbb R^n)$, cf.\ \cite[Thm.~9.3]{jakobsen18} and characterization \ref{thm:equiv_mild 2}, $\mathcal S_0'(\mathcal H)$ agrees exactly with those elements of $\mathcal S'(\mathcal H)$ which extend to bounded linear operators from $\mathcal S_0(\mathbb R^n)$ to $\mathcal S_0'(\mathbb R^n)$. Hence, among all the operators given by kernels/symbols in $\mathcal S'$, characterization \ref{thm:equiv_mild 4} and \ref{thm:equiv_mild 5} from the proposition yield characterizations through quantum harmonic analysis methods of those which are well-behaved in the Feichtinger setting.

\begin{remark}
We now briefly discuss how key properties of the mild distributions carry over to $\mathcal S_0'(\mathcal H)$ and $\mathcal S_0'(\mathbb R^{2n}) \oplus \mathcal S_0'(\mathcal H)$. First note that $\mathcal S_0'(\mathbb R^{2n}) \oplus \mathcal S_0'(\mathcal H)$ is not an algebra, but it is a module over $\mathcal S_0(\mathbb R^{2n}) \oplus \mathcal S_0(\mathcal H)$.
\begin{enumerate}[1)]
\item The space $\mathcal{S}_{0}^{Q}$ can through the Weyl quantization be identified with $\mathcal{S}_{0}(\mathbb{R}^{2n}) \oplus \mathcal{S}_{0}(\mathbb{R}^{2n})$ as a Banach space. As such, the Banach space dual $\mathcal{S}_{0}^{Q}{}'\coloneqq(\mathcal{S}_{0}^{Q})'$ of $\mathcal{S}_{0}^{Q}$ can be identified with $\mathcal{S}_{0}'(\mathbb{R}^{2n}) \oplus \mathcal{S}_{0}'(\mathbb{R}^{2n})$.
Hence bounded linear operators 
\[A\colon \mathcal{S}_{0}^{Q} \to \mathcal{S}_{0}^{Q}{}'\]
can be decomposed into four operators $A_{ij}\colon  \mathcal{S}_{0}(\mathbb{R}^{2n}) \to \mathcal{S}_{0}'(\mathbb{R}^{2n})$ for $i,j =1,2$. Each of these operators can be represented with an integral kernel $K_{ij} \in \mathcal{S}_{0}'(\mathbb{R}^{4n})$ by \cite[Thm.~9.3]{jakobsen18}.
Thus for $f_{1},f_{2},g_{1},g_{2} \in \mathcal{S}_{0}(\mathbb{R}^{2n})$ we can write $A$ as
\begin{multline*}
\langle A(f_{1}, A_{g_{1}}), (f_{2}, A_{g_{2}}) \rangle_{\mathcal{S}_{0}^{Q}{}', \mathcal{S}_{0}^{Q}} = \langle K_{11}, f_{2} \otimes f_{1} \rangle_{\mathcal{S}_0',\mathcal{S}_0} + \langle K_{21}, f_{2} \otimes g_{1} \rangle_{\mathcal{S}_0',\mathcal{S}_0}\\
+ \langle K_{12}, g_{2} \otimes f_{1} \rangle_{\mathcal{S}_0',\mathcal{S}_0} + \langle K_{22}, g_{2} \otimes g_{1} \rangle_{\mathcal{S}_0',\mathcal{S}_0}.
\end{multline*}
From here, it is not difficult to formulate a kernel theorem for bounded linear maps from $\mathcal{S}_0^Q$ to $\mathcal{S}_0^{Q}{}'$. We leave the details to the interested reader.
\item Recall that elements from $\mathcal S_0(\mathbb R^n)$ satisfy Poisson's summation formula. Indeed, there is a more general form with summation over different lattices, but for simplicity we stick to the following basic formulation: It is
\begin{align*}
    \sum_{k \in \mathbb Z^n} f(k) =(2\pi)^{n/2} \sum_{k \in 2\pi \mathbb{ Z}^n} \mathcal F(f)(k).
\end{align*}
In even dimensions, the right-hand side can of course be replaced by the symplectic Fourier transform. For $A = A_g \in \mathcal S_0(\mathcal{H})$, we therefore of course have
\begin{align*}
    \sum_{k \in 2\pi \mathbb Z^{2n}} \mathcal F_W(A_g)(k) = \sum_{k \in 2\pi \mathbb Z^{2n}} \mathcal F_\sigma(g)(k) = \frac{1}{(2\pi)^n}\sum_{k \in \mathbb Z^{2n}} g(k),
\end{align*}
which can be interpreted as a Poisson summation formula for quantum harmonic analysis. By Theorem \ref{thm:equivalent_feichtinger_characterizations} \ref{thm:equiv_feichtinger 2} we know that every $A \in \mathcal S_0(\mathcal{H})$ is of the form \[Af(s) = \int_{\mathbb R^n} f(t) K_A(s,t)\,\mathrm{d}y\]
for $K_A \in \mathcal{S}_0$. For $(x,\xi) \in \mathbb R^{2n}$, it is not hard to verify that the integral kernel of $AW_{(x, \xi)}$ is now
\begin{equation*}
K_{AW_{(x, \xi)}}(s, t) = e^{i\xi \cdot t+i\xi \cdot x/2} K_A(s,t+x).
\end{equation*}
By \cite[Cor.~3.15]{feichtinger_jakobsen2022} we have
\begin{align*}
    \mathcal F_W(A)(x, \xi) = \tr(AW_{(x, \xi)}) = \int_{\mathbb R^n} e^{i\xi \cdot s+i\xi \cdot x/2}K_A(s,s+x)\,\mathrm{d}s.
\end{align*}
On the other hand, if $A=A_g$ then $g$ can be computed by
\begin{align*}
    g(x,\xi) = \int_{\mathbb R^n} K_A\left(x+\frac{y}{2}, x-\frac{y}{2}\right) e^{- i \xi \cdot y}\,\mathrm{d}y \in \mathcal S_0(\mathbb R^{2n}).
\end{align*}
Hence, the above Poisson summation formula shows
\begin{align*}
   \sum_{k \in 2\pi\mathbb Z^{2n}} \mathcal F_W(A_g)(k)&=  \sum_{(k_1, k_2) \in \mathbb Z^{2n}} \int_{\mathbb R^n} e^{2\pi i k_2 \cdot s+2\pi^2 ik_1k_2}K_A(s,s+2\pi k_1)\,\mathrm{d}s=\frac{1}{(2\pi )^n} \sum_{k \in \mathbb Z^{2n}} g(k) \\
    &= \frac{1}{(2\pi )^n}\sum_{(k_1, k_2) \in \mathbb Z^{2n}} \int_{\mathbb R^n} K_A(k_1+\frac{y}{2}, k_1-\frac{y}{2}) e^{- i k_2 \cdot y}\,\mathrm{d}y.
\end{align*}
 A similar formula can be obtained by expressing the integral kernel in terms of the symbol.
\item Another point of view on the Poisson summation formula is the following: It is
\begin{align*}
  \frac{1}{(2\pi)^n} \sum_{k \in \mathbb Z^{2n}} \langle \delta_k, f\rangle = \frac{1}{(2\pi)^n} \sum_{k \in \mathbb Z^{2n}} f(k) = \sum_{k \in 2\pi \mathbb Z^{2n}} \widehat{f}(k) = \sum_{k \in 2\pi\mathbb Z^{2n}} \langle \delta_k, \widehat{f}\rangle,
\end{align*}
where $\sum_{k \in \mathbb Z^{2n}} \delta_k, \ \sum_{k \in 2\pi\mathbb Z^{2n}} \widehat{\delta_k}$ are in $S_0'(\mathbb R^{2n})$. Hence 
\begin{align*}
   \frac{1}{(2\pi)^n} \sum_{k \in \mathbb Z^{2n}} \alpha_k(\delta_0) = \sum_{k \in 2\pi\mathbb Z^{2n}} \mathcal F_\sigma(\alpha_k(\delta_0)) = \sum_{k \in 2\pi \mathbb Z^{2n}} \gamma_k \mathcal F_\sigma(\delta_0) = \sum_{k \in 2\pi \mathbb Z^{2n}} \frac{1}{(2\pi)^n}\gamma_k(1).
\end{align*}
Applying the inverse Fourier-Weyl transform, using that $\mathcal F_W^{-1}(\delta_0) = \mathrm{Id}$ and $\mathcal F_W^{-1}(1) = 2^nP$, we get 
\begin{align*}
    \sum_{k \in \mathbb Z^{2n}}\mathcal F_W^{-1} (\alpha_k(\delta_0)) &= \sum_{k \in \mathbb Z^{2n}} \gamma_k \mathcal F_W^{-1}(\delta_0) = \sum_{k \in \mathbb Z^{2n}} \gamma_k(\mathrm{Id}) = \sum_{k \in \mathbb Z^{2n}} W_{k},\\
    \sum_{k \in 2\pi \mathbb Z^{2n}} \mathcal F_W^{-1}(\gamma_k(1)) &= \sum_{k \in 2\pi \mathbb Z^{2n}} \alpha_k \mathcal F_W^{-1}(1) = 2^n\sum_{k \in 2\pi \mathbb Z^{2n}} \alpha_k(P) = 2^n\sum_{k \in2\pi  \mathbb Z^{2n}} W_{2k}P.
\end{align*}
Hence, we obtain in $\mathcal S_0'(\mathcal{H})$: 
\begin{align*}
   2^n \sum_{k \in 2\pi \mathbb Z^{2n}} W_{2k}P = \sum_{k \in \mathbb Z^{2n}} W_{k}.
\end{align*}
Applying this to some $A \in \mathcal S_0(\mathcal{H})$ gives the somewhat ominous result:
\begin{align*}
    \sum_{k \in \mathbb Z^{2n}} \mathcal F_W(A)(k) = 2^n\sum_{k \in 2\pi \mathbb Z^{2n}} \mathcal F_W(AP)(2k).
\end{align*}
This equality can conceptually be explained a little better by doing the following: Recall the parity operator is both symmetric and unitary, hence $P$ has spectrum contained in $\{ -1, 1\}$. The spectral decomposition of $\mathcal{H}$ is given by $\mathcal{H} = L_{\mathrm{even}}^2 \oplus L_{\mathrm{odd}}^2$, the decomposition into even and odd functions. Now, $P|_{L_{\mathrm{even}}^2} = \mathrm{Id}$ and $P|_{L_{\mathrm{odd}}^2} = -\mathrm{Id}$. We define a unitary square root $V$ of $P$ by letting $V|_{L_{\mathrm{even}}^2} = \mathrm{Id}$ and $V|_{L_{\mathrm{odd}}^2} = i\mathrm{Id}$. Then $VW_{(x,\xi)} = W_{(\xi, -x)}V$ and $W_z V = VW_{(-\xi, x)}$. Thus, we get
\begin{align*}
   \sum_{k \in 2\pi \mathbb Z^{2n}} \mathcal F_W(AP)(2k) = \sum_{k_1,k_2 \in 2\pi \mathbb Z^{n}} \mathcal F_W(V A V)(2k_1,-2k_1) = \sum_{k \in 4\pi  \mathbb Z^{2n}} \mathcal F_W(V A V)(k).
\end{align*}
Since passing from $A$ to $PAP$ plays the role of a reflection (analogously to passing from $f(z)$ to $f(-z)$, $z \in \mathbb C$), passing from $A$ to $VAV$ can be thought of as rotating the argument by $\frac{\pi}{2}$ (analogously to passing from $f(z)$ to $f(iz)$). Recalling now that taking the Fourier transform can be thought of as rotating the phase space by $\frac{\pi}{2}$, it is somewhat more reasonable that the Poisson summation formula should relate the $A$ and $V A V$.
\end{enumerate}

\end{remark}
\subsection{Miscellaneous examples}
We have thus far seen examples of quantum Segal algebras that are not ideals of $L^1 \oplus \mathcal T^1$, and will now describe a class of examples that are ideals. They are natural analogues of the Segal algebras from Example \ref{ex:segal} \ref{ex:segal E3}. 
\begin{example}Let $\mu$ be any Radon measure on $\mathbb R^{2n}$. Define the set $S_p^\mu$ by
\begin{align*}
    S_p^\mu = \{ (f, A) \in L^1 \oplus \mathcal T^1: ~\| \widehat{f}\|_{L^p(\mu)} + \| \widehat{A}\|_{L^p(\mathbb{R}^{2n},\mu)} < \infty\},
\end{align*}
with the norm 
\begin{align*}
    \| (f, A)\|_{S_p^\mu} = \| f\|_{L^1} + \| \widehat{f}\|_{L^p(\mathbb{R}^{2n},\mu)} + \| A\|_{\mathcal T^1} + \| \widehat{A}\|_{L^p(\mathbb{R}^{2n},\mu)}.
\end{align*}
Since $S_p^\mu$ contains any $(f, A) \in L^1 \oplus \mathcal T^1$ such that $\widehat{f}$ and $\widehat{A}$ both have compact support, it is a dense subspace of $L^1 \oplus \mathcal T^1$. Further, it is not hard to verify that $S_p^\mu$ is both a graded quantum Segal algebra and an ideal of $L^1 \oplus \mathcal T^1$. In the case that $\mu$ is a finite measure we have that $S_p^\mu = L^1 \oplus \mathcal T^1$, but for appropriately infinite measures $S_p^\mu$ is a strict subset.

Note that, as in \cite[(vii), p.~26]{reiter71}, these $S_p^\mu$ also have the following theoretical value: Given $A \in \mathcal T^1$ such that $\widehat{A}$ is not compactly supported, then one can explicitly write down a $\mu$ such that the graded quantum Segal algebra $S_1^\mu$ does not contain $A$.
\end{example}
We note that quantum Segal algebras, which are ideals of $L^1 \oplus \mathcal T^1$, are abstract Segal algebras in the sense of Burnham \cite[Def.~1.1]{Burnham1972}. In particular, such quantum Segal algebras have the same closed ideals as $L^1 \oplus \mathcal T^1$. We refrain from giving the precise result here and only refer to \cite[Thm.~1.1]{Burnham1972}. Instead, we want to give another result on such quantum Segal algebras, which is a close relative to our Corollary \ref{intersection_all_qsa}.
\begin{proposition}\label{intersection_all_qsa_ideals}
Let $G$ denote the set of all graded quantum Segal algebras which are ideals of $L^1(\mathbb R^{2n}) \oplus \mathcal T^1(\mathcal H)$. Then 
\[\mathcal{F}\Big(\bigcap_{S_1\oplus S_2\in G} S_1\oplus S_2\Big) = (C_{c}(\mathbb{R}^{2n})\cap \mathcal{F}_\sigma(L^1(\mathbb{R}^{2n}))) \oplus (C_{c}(\mathbb{R}^{2n})\cap \mathcal{F}_W(\mathcal T^1(\mathcal H))) .\]
\end{proposition}
\begin{proof}
As described above, the algebras $S_1^\mu$ can be utilized to show the inclusion ``$\subseteq$''. For the reverse inclusion, note that for any $f \in L^1(\mathbb R^{2n})$ with $\widehat{f} \in C_c(\mathbb R^{2n})$ we have that $(f, 0)$ is contained in the intersection by Corollary \ref{intersection_all_qsa}. Let $A \in \mathcal T^1$ such that $\widehat{A}$ is compactly supported. Then, there exists $g \in \mathcal S(\mathbb R^{2n})$ such that $\widehat{g}$ is compactly supported and $\widehat{g} \equiv 1$ on the support of $\widehat{A}$. In particular, $(g, 0)$ is contained in every graded quantum Segal algebra. Thus, the convolution $(g, 0) \ast (0, A)$ is contained in every graded quantum Segal algebra which is an ideal in $L^1 \oplus \mathcal T^1$. By comparing Fourier transforms, we find that $(g, 0) \ast (0, A) = (0, A)$. Therefore, the element $(0, A)$ is contained in every graded quantum Segal algebra which is an ideal of $L^1 \oplus \mathcal T^1$.
\end{proof}

We still owe the reader an example justifying the special treatment of graded quantum Segal algebras, i.e., showing that not every quantum Segal algebra is graded. Here it is:
\begin{example}\label{qsa:notgraded}
    Assume that $S_1$ and $S_2$ are two Segal algebras contained in the Feichtinger algebra $\mathcal S_0$ where there exist $f_1\in S_1$ and $g_1\in S_2$ such that $f_1\not \in S_2$ and $g_1\not\in S_1$. 
     Let 
    \[S_T=\{(f,A_{\tilde{g}}): f+g \in S_1, ~f-g \in S_2\}\]
    with the norm 
   \[\|(f,A_{\tilde{g}})\|_{S_T}\coloneqq \|f+g\|_{S_1}+\|f-g\|_{S_2}.\]
   Then 
   \begin{align*}
      \|(f,A_{\tilde{g}})\ast(h,A_{\tilde{j}})\|_{S_T} &= \|(f+g)\ast(h+j)\|_{S_1}+\|(f-g)\ast(h-j)\|_{S_2}\\
       &\le \|(f+g)\|_{S_1}\|(h+j)\|_{S_1}+\|(f-g)\|_{S_2}\|(h-j)\|_{S_2}\\
       &\le \|(f,A_{\tilde{g}})\|_{S_T}\|(h,A_{\tilde{j}})\|_{S_T},
   \end{align*}
   showing that $S_T$ is a normed algebra. Further, it is not hard to see that it is complete and satisfy \ref{label:QS3} and \ref{label:QS4}. The property \ref{label:QS1} follows from the fact that $C_0(\mathbb{R}^{2n})\oplus C_0(\mathbb{R}^{2n})\subset \mathcal{F}(S_T)$. 
   If we set $f+g=f_1$ and $f-g=g_1$ we have an example $(f, A_g) \in S_T$ where $(f,0)$ cannot be in $S_T$.

   We give such an example of two Segal algebras $S_1$, $S_2$ in the case $n=1$, and leave it to the reader to carry out the necessary adaptations for obtaining such subalgebras of $\mathcal S_0(\mathbb R^{2n})$.
   Let $S_1=\mathcal S_0(\mathbb R^2) \cap \mathcal{F}^{-1}(L^2(\mathbb{R}^2, (1+x^2)^{10}))$ and $S_2=\mathcal S_0(\mathbb R^2) \cap \mathcal{F}^{-1}(L^2(\mathbb{R}^2, e^{2x}))$. By Example \ref{ex:segal}, these are intersections of Segal algebras, hence themselves Segal algebras.
    Consider the function \[f_a(x,\xi)=\frac{1}{x^2+a^2}\frac{1}{\xi^2+a^2}\]
    with Fourier transform 
    \[\mathcal{F}(f_a)(x,\xi)=\frac{\pi}{2a^2}e^{-a(|x|+|\xi|)}.\]
    Using that $f_a\in \mathcal S_0(\mathbb R^{2})$ for all $a$, we observe that $f_1\in S_1$ but $f\not \in S_2$. On the other hand, the function \[g(x,\xi)=\frac{1-e^{-x^2}}{1+x^2}\frac{1-e^{-\xi^2}}{1+\xi^2}H(-x)H(-\xi),\] with $H$ the Heaviside function, is smooth and in $L^1(\mathbb{R}^2)$ with the first two derivatives in $L^1(\mathbb{R}^2)$. Hence, its Fourier transform is in $L^1(\mathbb{R}^2)$ as well. We know that $\mathcal S_0(\mathbb R^2)\ast L^1(\mathbb R^2) \subset \mathcal S_0(\mathbb R^2)$, hence \[g_1=\mathcal{F}^{-1}(f_1)\ast \mathcal{F}^{-1}\left(g\right)\in \mathcal S_0(\mathbb R^2).\]
    Then $g_1\in S_2$ while $g_1\not \in S_1$.

\end{example}

\bibliographystyle{abbrv}
\bibliography{main}
\vspace{1cm}
\begin{multicols}{2}
[
\noindent
Eirik Berge\\
\href{ebbeberge94@gmail.com}{\Letter ~ebbeberge94@gmail.com}
]

\noindent
Stine Marie Berge\\
\href{stine.berge@math.uni-hannover.de}{\Letter ~stine.berge@math.uni-hannover.de}
\\
\noindent
Institut f\"{u}r Analysis\\
Leibniz Universit\"at Hannover\\
Welfengarten 1\\
30167 Hannover\\
GERMANY\\ 

\noindent
Robert Fulsche\\
\href{fulsche@math.uni-hannover.de}{\Letter ~fulsche@math.uni-hannover.de}
\\
\noindent
Institut f\"{u}r Analysis\\
Leibniz Universit\"at Hannover\\
Welfengarten 1\\
30167 Hannover\\
GERMANY
\end{multicols}

\end{document}